\documentclass[10pt,a4paper]{article}
\usepackage{graphicx}
\usepackage{verbatim}
\usepackage{amsthm}
\usepackage{mathrsfs}

\newtheorem{theorem}{Theorem}[section]
\newtheorem{lemma}[theorem]{Lemma}
\newtheorem{proposition}[theorem]{Proposition}
\newtheorem{corollary}[theorem]{Corollary}
\theoremstyle{remark}
\newtheorem{remark}[theorem]{Remark}
\theoremstyle{definition}
\newtheorem{definition}[theorem]{Definition}
\theoremstyle{definition}
\newtheorem{example}[theorem]{Example}
\theoremstyle{notation}

\newcommand{\fxy}{\mathrm{(}\mathrm{F}_{X,Y}\mathrm{)}}
\newcommand{\fgl}{\mathrm{(}\mathrm{F}_{\mathrm{Gl}}\mathrm{)}}
\newcommand{\fmp}{\mathrm{(}\mathrm{F}_{\mathrm{MP}}\mathrm{)}}

\newcommand{\norm}[1]{\left\Vert#1\right\Vert}
\newcommand{\abs}[1]{\left\vert#1\right\vert}
\newcommand{\set}[1]{\left\{#1\right\}}

\newcommand{\nhat}{\hat{\mathbb{N}}}

\newcommand{\Pm}{\mathrm{Prim}}

\newcommand{\Pme}{\mathrm{Prime}}
\newcommand{\MP}{\mathrm{Min\mbox{-}Primal}}
\newcommand{\Gl}{\mathrm{Glimm}}
\newcommand{\Fac}{\mathrm{Fac}}
\newcommand{\mint}{\otimes_{\alpha}}
\newcommand{\maxt}{\otimes_{\max}}
\newcommand\restr[2]{{% we make the whole thing an ordinary symbol
  \left.\kern-\nulldelimiterspace % automatically resize the bar with \right
  #1 % the function
  \right|_{#2} % this is the delimiter
  }}
  
\title{Exact C$^{\ast}$-algebras and $C_0(X)$-structure}

\author{David McConnell\footnote{This work was supported
by the Science Foundation Ireland under grant 11/RFP/MTH3187.}\\
\textit{School of Mathematics, Trinity College, Dublin 2, Ireland} \\ \texttt{mcconnd@tcd.ie} 
}

\usepackage[all]{xy}
\usepackage{amssymb}

\usepackage{amsmath}
\usepackage{fullpage}

\begin{document}

\maketitle

\begin{abstract}
We study tensor products of a $C_0 (X)$-algebra $A$ and a $C_0 (Y)$-algebra $B$, and analyse the structure of their minimal tensor product $A \mint B$ as a $C_0 (X \times Y)$-algebra.  We show that when $A$ and $B$ define continuous C$^{\ast}$-bundles, that continuity of the bundle arising from the $C_0 (X \times Y)$-algebra $A \mint B$ is a strictly weaker property than continuity of the `fibrewise tensor products' studied by Kirchberg and Wassermann.  For a fixed quasi-standard C$^{\ast}$-algebra $A$, we show that $A \mint B$ is quasi-standard for all quasi-standard $B$ precisely when $A$ is exact, and exhibit some related equivalences.
\end{abstract}
\section{Introduction}

Bundles or fields of C$^{\ast}$-algebras have long been an important aspect of the topological decomposition theory of C$^{\ast}$-algebras.  For instance the  Gelfand-Naimark theorem may be seen as representing an arbitrary commutative C$^{\ast}$-algebra as the section algebra of a bundle over its character space, with one-dimensional fibres.  In the general (non-commutative) case, many analogous constructions have been studied.  Often these constructions give rise to bundles of a very general type, and the structure of such bundles is in general difficult to determine. There are however many classes of C$^{\ast}$-algebras for which a well-behaved bundle representation may be obtained. In this work we are concerned with the stability of such classes of C$^{\ast}$-bundles under the operation of taking tensor products, in particular with respect to the minimal (or spatial) tensor norm.

A  particular case of interest in the theory of C$^{\ast}$-bundles is that where the base space coincides with the topological space of Glimm (or Minimal Primal) ideals of the bundle algebra.  This approach has its origins in~\cite{fell}, where it was shown that a C$^{\ast}$-algebra with Hausdorff primitive ideal space admits a representation as a continuous C$^{\ast}$-bundle (or a maximal full algebra of operator fields, using Fell's terminology) over this space.  The work of Dauns and Hofmann in~\cite{dauns_hofmann} gave a bundle representation theory that was valid for an arbitrary C$^{\ast}$-algebra, at the possible expense of continuity of the norm function of a section.

The work of Archbold and Somerset~\cite{arch_som_qs} identified a class of C$^{\ast}$-algebras with well-behaved Dauns-Hofmann representation, the so-called \emph{quasi-standard} C$^{\ast}$-algebras.  This class represents a significant weakening of the condition that the C$^{\ast}$-algebra has Hausdorff primitive spectrum, while preserving many of the nice properties associated with this stronger condition.  Indeed, all von Neumann algebras are quasi-standard~\cite{glimm}, as are   the `local multiplier algebras' studied extensively by Ara and Mathieu~\cite[\S 2.5]{ara_mat_loc} and many group C$^{\ast}$-algebras~\cite{kaniuth_groups}. 

 Tensor products of continuous bundles of C$^{\ast}$-algebras are known to exhibit pathological behaviour.  The earliest examples of this were given by Kirchberg and Wassermann in~\cite{kirch_wass}, who showed that continuity of a C$^{\ast}$-bundle was in general not preserved  by tensoring fibrewise with a fixed C$^{\ast}$-algebra $B$.  Moreover, it was shown that such an operation preserves continuity for all C$^{\ast}$-bundles precisely when $B$ is exact.  Archbold later obtained a localisation of this result in~\cite{arch_cb}, where continuity at a point was characterised in terms of a weaker exactness-type condition.  Similar questions have been studied extensively by Blanchard and coauthors in~\cite{blanch_def},~\cite{blanch_cx},~\cite{blanch_free},~\cite{blanch_exact} and~\cite{blanch_wass_exact}.

In~\cite{m_glimm}, we constructed the Glimm ideal space of the minimal tensor product of two C$^{\ast}$-algebras in terms of those of the factors. As a consequence, it is possible to construct the Dauns-Hofmann bundle of $A \mint B$ in terms of that of $A$ and $B$, although the fibre algebras of this bundle remain difficult to describe without additional assumptions on $A$ and $B$.  In particular, it is not immediate from our results in~\cite[\S5]{m_glimm} whether or not this bundle agrees with the fibrewise tensor product of Kirchberg and Wassermann.   Moreover, it remains difficult to show in general whether or not certain classes of C$^{\ast}$-algebras with well-behaved Dauns-Hofmann representations, for example the quasi-standard C$^{\ast}$-algebras, are stable under minimal tensor products.

The notion of a $C_0 (X)$-algebra, introduced by Kasparov~\cite{kasparov}, is closely related to that of a C$^{\ast}$-bundle.  In some sense $C_0 (X)$-algebras generalise the Dauns-Hofmann construction to give a bundle representation of a C$^{\ast}$-algebra $A$ over any locally compact Hausdorff space $X$ that is a continuous image of the primitive ideal space of $A$. In section~\ref{s:tensorbundle} we study a natural construction which equips the minimal tensor product $A \mint B$ of a $C_0 (X)$-algebra $A$ and a $C_0 (Y)$-algebra $B$ the structure of a $C_0 (X \times Y)$-algebra.  It has been observed previously that this bundle representation of $A \mint B$ may differ from the fibrewise tensor product~\cite{kirch_wass},~\cite{blanch_exact}, and that exactness of $A$ or $B$ plays a decisive role in these considerations.

 While the results of Kirchberg, Wassermann and Archbold give necessary and sufficient conditions for the continuity of the fibrewise minimal tensor product, less is known regarding the tensor product bundle that we study in this work.  Indeed, we show in section~\ref{s:ineq} that there are quasi-standard C$^{\ast}$-algebras $A$ and $B$ such that $A \mint B$ is quasi-standard, while the fibrewise tensor product of $A$ and $B$ gives rise to a discontinuous C$^{\ast}$-bundle.  
 
 Related work of Blanchard~\cite{blanch_exact} (concerning the amalgamated $C_0 (X)$-tensor product of two $C_0 (X)$-algebras $A$ and $B$) indicated that continuity may fail for the tensor product bundle that we define in \S~\ref{s:tensorbundle} also.  However, the argument used in~\cite[\S 3]{blanch_exact} relies on specific properties of the C$^{\ast}$-algebras involved.  We show in \S~\ref{s:exact} that  for an inexact continuous $C_0 (X)$-algebra $A$, one can always construct a continuous $C_0 (X)$-algebra $B$ such that $A \mint B$ is discontinuous as a $C_0 (X \times Y)$-algebra.  As a consequence it is shown in Theorem~\ref{t:exact} that stability of continuity is in fact equivalent to exactness of $A$.  Thus our tensor product construction identifies exactness in precisely the same way as the fibrewise tensor product of Kirchberg and Wassermann~\cite[Theorem 4.5]{kirch_wass}.

In \S 6 we investigate the question of stability of the property of quasi-standardness under the operation of taking tensor products (in particular with respect to the minimal C$^{\ast}$-norm).  One consequence of this is the observation that, in general, the C$^{\ast}$-bundle associated with the Dauns-Hofmann representation of such a tensor product is not, in general, given by the fibrewise tensor product of the corresponding bundles of the factors. 

Until now, it appears that there were no known examples of a pair of quasi-standard C$^{\ast}$-algebras whose minimal tensor product fails to be quasi-standard.  It was shown by Kaniuth in~\cite{kaniuth} that if $A \mint B$ satisfies Tomiyama's property (F), then $A \mint B$ is quasi-standard if and only if $A$ and $B$ are quasi-standard.  In particular this is the case whenever either $A$ or $B$ is exact.  The assumption of property (F) was weakened in~\cite{m_glimm} to an assumption involving exact sequences related to Glimm ideals.

 We show in Theorem~\ref{t:qs} that if $A$ is a quasi-standard C$^{\ast}$-algebra which is not exact, then one can always construct a quasi-standard C$^{\ast}$-algebra $B$ for which $A \mint B$ is not quasi-standard.   In particular it follows that a quasi-standard C$^{\ast}$-algebra $A$ is exact if and only if $A \mint B$ is quasi-standard for all quasi-standard $B$.  This is consistent with the characterisation of exactness obtained by Kirchberg and Wassermann in~\cite{kirch_wass}, though perhaps surprising in light of the results of section~\ref{s:ineq}.  Similarly, in the unital case we show in Theorem~\ref{t:maxt} that stability of the property of quasi-standardness under taking maximal tensor products is equivalent to nuclearity.

\section{Preliminaries on $C_0(X)$-algebras and C$^{\ast}$-bundles}
We begin with the basic definitions and properties of $C_0 (X)$-algebras and C$^{\ast}$-bundles, and list some well-known equivalences between the two.    Further we will introduce the relevant (topological) spaces of ideals that arise as natural candidates for the base space $X$ in this context, namely the spaces of Glimm and minimal primal ideals. A comprehensive introduction to the theory may be found in~\cite[Appendix C]{williams}, see also~\cite{nilsen_bundles}.

Throughout this work, $C_0 (X)$ will denote the algebra of continuous complex-valued functions vanishing at infinity on the locally compact Hausdorff space $X$.  For a C$^{\ast}$-algebra $A$, $\Pm (A)$ (respectively $\Fac (A)$) will denote the space of kernels of irreducible (respectively factorial) representations of $A$, and will be considered as a topological space with its usual hull-kernel topology.  Unless otherwise specified, an ideal of a C$^{\ast}$-algebra will mean a proper, closed two-sided ideal.    We will denote by $Z(A)$ the centre of $A$ and by $M(A)$ the multiplier algebra of $A$.
\begin{definition}
A C$^{\ast}$-bundle is a triple $\mathscr{A} = (X , A , \pi_x : A \rightarrow A_x )$ where $X$ is a locally compact Hausdorff space, $A$ a C$^{\ast}$-algebra, and $\pi_x : A \rightarrow A_x$  surjective $\ast$-homomorphisms for all $x \in X$ satisfying
\begin{enumerate}
\item[(i)] the family $\{ \pi_x : x \in X \}$ is faithful, i.e., $\bigcap_{x \in X} \ker ( \pi_x ) = \{ 0 \}$, and
\item[(ii)] for each $f \in C_0 (X)$ and $a \in A$ there is an element $f \cdot a \in A$ with the property that 
\[
\pi_x ( f \cdot a ) = f(x) \pi_x (a) \mbox{ for all } x \in X.
\]
\end{enumerate}
If in addition the functions $N(a): X \rightarrow \mathbb{R}_+$, $x \mapsto \| \pi_x (a) \|$ where $a \in A$, belong to $C_0(X)$ for all $a \in A$ then we say that $\mathscr{A}$ is a \emph{continuous C$^{\ast}$-bundle} over $X$.  If for all $a \in A$ the functions $N(a)$ are upper-semicontinuous (resp. lower-semicontinuous) on $X$, and if for each $\varepsilon > 0$  the set $\{ x \in X: N(a) \geq \varepsilon \}$ has compact closure in $X$, then we say that $\mathscr{A}$ is an \emph{upper-semicontinuous C$^{\ast}$-bundle} (resp. lower-semicontinuous C$^{\ast}$-bundle).
\end{definition}
\begin{remark}
There are several distinct (but often equivalent) definitions of C$^{\ast}$-bundles throughout the literature. Our definition of a continuous C$^{\ast}$-bundle was first used by Kirchberg and Wassermann in~\cite{kirch_wass}, and later by Archbold~\cite{arch_cb} and Blanchard and Wassermann~\cite{blanch_wass_exact}.  This is in turn equivalent to the definition of a `maximal algebra of operator fields' in the sense of Fell~\cite{fell}, also used in~\cite{dix},~\cite{apt},~\cite{lee} and~\cite{arch_som_qs}.  

While C$^{\ast}$-bundles with semicontinuous norm functions are discussed in~\cite{kirch_wass} and~\cite{arch_cb}, no precise definition of a semicontinuous C$^{\ast}$-bundle is given there.  In the upper-semicontinuous case, our definition is easily seen to be equivalent to that of Rieffel~\cite{rieffel}, also considered in~\cite{nilsen_bundles} and~\cite{williams}.
\end{remark}

\begin{definition}
Let $A$ be a C$^{\ast}$-algebra and $X$ a locally compact Hausdorff space.  We say that $A$ is a \emph{$C_0(X)$-algebra} if  there is a $\ast$-homomorphism $\mu_A : C_0 (X) \rightarrow ZM(A)$ with the property that $\mu_A (C_0 (X) )A = A$. 
\end{definition}

It follows from the  Dauns-Hofmann Theorem~\cite{dauns_hofmann} which we will discuss below (Theorem~\ref{t:dh}), that there is a $\ast$-isomorphism $\theta_A : C^b ( \Pm (A) ) \rightarrow ZM(A)$ with the property that
\begin{equation}
\label{e:dh}
\theta_A (f) a + P = f(P) (a + P ), \mbox{ for } a \in A , f \in C^b ( \Pm (A)), P \in \Pm (A).
\end{equation}
This gives an equivalent definition of a $C_0 (X)$-algebra: a C$^{\ast}$-algebra $A$ is a $C_0 (X)$-algebra if and only if there exists a continuous map $\phi_A : \Pm (A) \rightarrow X$.  The maps $\mu_A$ and $\phi_A$ are related via $ \mu_A (f) = \theta_A ( f \circ \phi_A )$ for all $f \in C_0 (X)$~\cite[Proposition C.5]{williams}.

For clarity we will denote any $C_0 (X)$-algebra $A$ by the triple $(A , X , \phi_A)$ or $(A , X , \mu_A )$.  For  $x \in X$ we define the ideal $I_x$ via
\[
I_x = \mu_A \left( \{ f \in C_0 (X) : f(x) = 0 \} \right) A = \bigcap \{ P \in \Pm (A) : \phi_A (P) = x \},
\]
see~\cite[Section 2]{nilsen_bundles} for example.

The relationship between $C_0 (X)$-algebras and C$^{\ast}$-bundles is well known, see~\cite{nilsen_bundles} or~\cite[Appendix C]{williams} for example.  We give details in the following proposition, which will be used frequently in what follows.
\begin{proposition}
\label{p:c0xbundles}
Let $A$ be a C$^{\ast}$-algebra and $X$ a locally compact Hausdorff space.
\begin{enumerate}
\item[(i)] If $(A,X,\mu_A )$ is a $C_0 (X)$-algebra, then with $A_x = A/ I_x$ and $\pi_x : A \rightarrow A_x $ the quotient $\ast$-homomorphism, the triple $(X,A, \pi_x : A \rightarrow A_x )$ is an upper-semicontinuous C$^{\ast}$-bundle~\cite[Theorem 2.3]{nilsen_bundles}.
\item[(ii)] If $(X,A,\pi_x :A \rightarrow A_x )$ is a C$^{\ast}$-bundle, then setting $\mu_A (f) a = f \cdot a$ for $f \in C_0 (X),a \in A$ defines a $\ast$-homomorphism $\mu_A : C_0 (X) \rightarrow ZM(A)$ such that $(A,X,\mu_A)$ is a $C_0 (X)$-algebra.   Moreover, $(X,A,\pi_x :A \rightarrow A_x )$ is an upper-semicontinuous C$^{\ast}$-bundle if and only if $\ker (\pi_x ) = I_x $ for all $x \in X$~\cite[Lemmas 2.1 and 2.3]{kirch_wass}.
\item[(iii)] The $C_0 (X)$-algebra $(A,X,\mu_A )$ gives rise to a continuous C$^{\ast}$-bundle $(X,A,\pi_x : A \rightarrow A_x )$ if and only if the corresponding base map $\phi_A : \Pm (A) \rightarrow X$ is an open map~\cite[Theorem 5]{lee}.
\end{enumerate}
\end{proposition}

As a consequence of Proposition~\ref{p:c0xbundles}, we will regard $C_0 (X)$-algebras and upper-semicontinuous C$^{\ast}$-bundles as being (essentially) equivalent.  Moreover, we may unambiguously speak of a $C_0 (X)$-algebra $(A,X,\mu_A)$ being \emph{continuous} if the corresponding C$^{\ast}$-bundle $(X,A,\pi_x: A \rightarrow A_x )$ is continuous.

If $A$ is a C$^{\ast}$-algebra such that $\Pm (A)$ is Hausdorff, then it is clear that the triple $(A, \Pm (A) , \theta_A )$, where $\theta_A$ is the Dauns-Hofmann $\ast$-isomorphism~(\ref{e:dh}), is a continuous $C_0 ( \Pm (A) )$-algebra.  Moreover, the fibre algebras of the corresponding C$^{\ast}$-bundle are necessarily simple in this case.  This construction was first studied by Fell~\cite{fell}.  However, the class of C$^{\ast}$-algebras $A$  with $\Pm (A)$  Hausdorff is somewhat narrow.    The Dauns-Hofmann Theorem~\cite{dauns_hofmann} shows that any C$^{\ast}$-algebra $A$ admits a representation as an upper-semicontinuous C$^{\ast}$-bundle over its space of \emph{Glimm ideals} (or a closely related space), which we define below.

For a C$^{\ast}$-algebra $A$, define an equivalence relation $\approx$ on $\Pm (A)$ as follows: for $P,Q \in \Pm (A)$, $P \approx Q$ if and only if $f(P) = f(Q)$ for all $f \in C^b ( \Pm (A) )$.  As a set, we define $\Gl (A)$ as the quotient space $\Pm (A) / \approx$, and we denote by $\rho_A : \Pm (A) \rightarrow \Gl (A)$ the quotient map.

For $f \in C^b ( \Pm (A))$, define $f^{\rho} : \Gl (A) \rightarrow \mathbb{C} $ as follows: $f^{\rho} \circ \rho_A (P) = f(P)$, which is well defined by definition of $\approx$.  We equip $\Gl (A)$ with the topology $\tau_{cr}$ induced by the functions $\{ f^{\rho} : f \in C^b ( \Pm (A) ) \}$, so that $(\Gl (A) , \tau_{cr} )$ is a completely regular Hausdorff space, and $f \mapsto f^{\rho}$ is a $\ast$-isomorphism of $C^b ( \Pm (A) )$ onto $C^b ( \Gl (A) )$, see~\cite[Theorem 3.9]{gill_jer} or~\cite[p. 351]{arch_som_qs}.  Note that the map $\rho_A : \Pm (A) \rightarrow \Gl (A)$ is continuous.

To each $\approx$-equivalence class $p = [P ] \in \Gl (A)$ we assign a closed two-sided ideal $G_p = \bigcap \{ P': P' \approx P \} = \bigcap \{ P': \rho_A (P') = p \}$, the  Glimm ideal of $A$ corresponding to $p$.

We wish to represent $A$ as a C$^{\ast}$-bundle over $\Gl (A)$, or a $C_0 ( \Gl (A) )$-algebra. In general however, it may happen that $\Gl (A)$ is not locally compact.  Being a completely regular Hausdorff space, $\Gl (A)$ has a homeomorphic embedding into its Stone-\v{C}ech compactification $\beta \Gl (A)$.

\begin{theorem}({J. Dauns, K.H. Hofmann)}
\label{t:dh}
Let $A$ be a C$^{\ast}$-algebra.  Then $A$ is a $C_0 (Y)$-algebra, and hence the section algebra of an upper-semicontinuous C$^{\ast}$-bundle $(Y , A , \pi_p : A \rightarrow A_p )$, where
\begin{enumerate}
\item[(i)] if $\Gl (A)$ is locally compact, $Y = \Gl (A)$, $A_p = A/ G_p$ and $\pi_p = q_p : A \rightarrow A / G_p $ the quotient $\ast$-homomorphism for all $p \in \Gl (A)$,
\item[(ii)] if $\Gl (A)$ is not locally compact, $Y = \beta \Gl (A)$, and
\begin{itemize}
\item for $p \in \Gl (A) $,  $A_p = A/ G_p$ and $\pi_p = q_p : A \rightarrow A / G_p $ the quotient $\ast$-homorphism, and
\item for $p \in \beta \Gl (A) \backslash \Gl (A)$, $A_p = \{ 0 \}$.
\end{itemize}

\end{enumerate}
\end{theorem}

If $\Pm (A)$ is Hausdorff, then being locally compact, it is necessarily completely regular.  Thus $\Pm (A) = \Gl (A)$, both as sets of ideals and topologically.  In this case the Dauns-Hofmann bundle associated with $A$ is precisely the continuous C$^{\ast}$-bundle over $\Pm (A)$ obtained by Fell~\cite[Theorem 2.3]{fell}.  More generally, Lee's theorem~\cite[Theorem 4]{lee} implies that the Dauns-Hofmann bundle of a C$^{\ast}$-algebra $A$ is a continuous C$^{\ast}$-bundle if and only if the complete regularisation map is open.  Note that if this is the case then necessarily $\Gl(A)$ is locally compact.

\begin{definition}
An ideal $I$ of a C$^{\ast}$-algebra $A$ is said to be \emph{primal} if given $n \geq 2$ and ideals $J_1, \ldots ,J_n$ of $A$ such that $J_1J_2 \ldots J_n = \{ 0 \}$, then there is an index $1 \leq i \leq n$ with $J_i \subseteq I$.  A C$^{\ast}$-algebra $A$ is called \emph{quasi-standard} if
\begin{enumerate}
\item[(i)] the Dauns-Hofmann representation of $A$ is a continuous C$^{\ast}$-bundle over $\Gl (A)$, and
\item[(ii)] For each $p \in \Gl (A)$, the Glimm ideal $G_p$ is a primal ideal of $A$.
\end{enumerate}
\end{definition}

For a C$^{\ast}$-algebra $A$ we will denote by $\MP (A)$ its set of minimal (w.r.t. inclusion) primal ideals.  The canonical topology $\tau$ on $\MP (A)$ is the weakest topology such that the norm functions $I \mapsto \norm{ a + I}$ on $\MP (A)$ are continuous for all $a \in A$.  If $A$ is a C$^{\ast}$-algebra for which every $G \in \Gl (A)$ is a primal ideal of $A$, then necessarily we have $\Gl (A) = \MP (A)$ as sets.  From~\cite[Theorem 3.3]{arch_som_qs}, a C$^{\ast}$-algebra $A$ is quasi-standard if and only if $(\Gl (A) , \tau_{cr} ) = ( \MP (A) , \tau )$ (i.e., as sets of ideals and topologically).

There are other equivalent definitions of quasi-standard C$^{\ast}$-algebras, see~\cite[\S 3 and \S 4]{arch_som_qs} for example.  Since primitive ideals are primal, the class of quasi-standard C$^{\ast}$-algebras properly contains the class of C$^{\ast}$-algebras $A$ with $\Pm(A)$ Hausdorff.  In fact, for separable $A$, quasi-standardness of $A$ is equivalent to the condition that $A$ is a continuous $C_0 (X)$-algebra such that there exists a dense subset $D$ of $X$ with $I_x$ primitive for all $x \in D$.

The Glimm space of a C$^{\ast}$-algebra appears as an intermediate step in any representation of a C$^{\ast}$-algebra as a $C_0 (X)$-algebra, due to a certain universal property of the complete regularisation of a topological space. If $X$ is a completely regular space and $\phi : \Pm (A) \rightarrow X$ a continuous map, then $\phi$ induces a continuous map $\psi : \Gl (A) \rightarrow X$ with $\phi = \psi \circ \rho_A$, i.e.,
\[
\xymatrix{
\Pm (A) \ar^{\phi}[rd] \ar^{\rho_A}[d] & \\
\Gl (A) \ar^{\psi}[r] & X
}
\]
commutes.  Conversely, starting with a continuous map $\psi : \Gl (A) \rightarrow X$, we may set $\phi = \psi \circ \rho_A$, so that $\phi : \Pm (A) \rightarrow X$ is continuous.

If in addition $X$ is locally compact, then $A$ is a $C_0 (X)$ algebra if and only if there is a continuous map $\psi_A : \Gl (A) \rightarrow X$. This fact is useful when working with tensor products of C$^{\ast}$-algebras, since by~\cite{m_glimm} we may always construct $\Gl (A \mint B )$ in terms of $\Gl (A)$ and $\Gl (B)$.  The same is not true in general for the spaces $\Pm (-)$ and $\Fac (-)$.

In the remainder of this section we give some technical results on the structure of $C_0 (X)$-algebras and Glimm spaces of C$^{\ast}$-algebras which we will make reference to in subsequent sections.

\begin{lemma}
\label{l:capglimm}
Let $A$ be a $C_0 (X)$-algebra with base map $\phi_A : \mathrm{Prim}(A) \rightarrow X$, and denote by $\psi_A : \mathrm{Glimm}(A) \rightarrow X$ the induced continuous map with the property that $\psi_A \circ \rho_A = \phi_A$.  Then for each $x \in X$,
\[
I_x = \bigcap \{ G \in \mathrm{Glimm}(A ) : \psi_A (G) = x \}.
\]
\end{lemma}
\begin{proof}
Denote by $F = \psi_A^{-1} ( \{ x \} ) \subseteq \mathrm{Glimm} (A)$, so that for $P \in  \mathrm{Prim}(A)$, we have $\phi_A (P) = x $ if and only if $\rho_A (P) \in F$.  Thus
\begin{eqnarray*}
I_x & = & \bigcap \{ P \in \mathrm{Prim}(A) : \rho_A (P) \in F \} \\
& = & \bigcap_{p \in F} \left\lbrace \bigcap \{ P \in \mathrm{Prim} (A) : \rho_A (P) = p \} \right\rbrace \\
& = & \bigcap \{ G \in \mathrm{Glimm} (A) : \psi_A (G) = x \}.
\end{eqnarray*}
\end{proof}

As was the case in~\cite{kaniuth} and~\cite{m_glimm}, we will make use of the fact that the space $\Gl (A)$ may equivalently be constructed (both as a set of ideals and topologically) as complete regularisation of $\Fac (A)$~\cite{kaniuth}.  This is useful when working with tensor products, since there is a continuous retraction of the canonical (homeomorphic) embedding of $\Fac (A) \times \Fac (B)$ into $\Fac (A \mint B )$.  For $I,J \in \Fac (A)$ we will write $I \approx_f J$ to denote the equivalence relation $f(I) = f(J)$ for all $f \in C^b ( \Fac (A) )$, and denote by $\rho_A^f : \Fac (A) \rightarrow \Gl (A)$ the complete regularisation map, so that the restrictions of each to $\Pm (A)$ satisfy $\restr{\approx}{\Pm (A)} = \approx$ and $\restr{\rho_A^f}{\Pm (A)} = \rho_A$.

Lemma~\ref{l:phifac} and Proposition~\ref{p:phifac} below generalise parts of~\cite[lemmas 2.1 and 2.2]{kaniuth} from Glimm ideals to the ideals $I_x$ defined in the $C_0 (X)$-algebra case.

\begin{lemma}
\label{l:phifac}
Let $A$ be a C$^{\ast}$-algebra and $X$ a locally compact Hausdorff space.  Then any continuous map $\phi_A : \mathrm{Prim}(A) \rightarrow X$ has a unique extension to a continuous map $\phi_A^f : \mathrm{Fac}(A) \rightarrow X$. For $I \in \mathrm{Fac}(A)$ and $P \in \mathrm{hull} (I)$ we have $\phi_A^f (I) = \phi_A (P)$.
\end{lemma}
\begin{proof}
By~\cite[Lemma 3.1]{lazar_tensor} there is a unique continuous map $\tilde{\phi}_A : \mathrm{prime}(A) \rightarrow X$ extending $\phi_A$.  Set $\phi_A^f = \restr{\tilde{\phi}_A}{\mathrm{Fac}(A)}$, then $\phi_A^f$ is continuous since $\tilde{\phi}_A$ is.  Uniqueness follows from the fact that $\Pm (A)$ is dense in $\Fac (A)$.

Take $I \in \mathrm{Fac}(A)$ and $P \in \mathrm{hull} (I)$, so that $P \in \mathrm{cl} \{ I \}$. Then since  $\phi_A^f$ is continuous and extends $\phi_A$, we necessarily have that $\phi_A^f (I) = \phi_A^f (P) = \phi_A (P)$.
\end{proof}
\begin{proposition}
\label{p:phifac}
Let $A$ be a C$^{\ast}$-algebra and $X$ a locally compact Hausdorff space.  
\begin{enumerate}
\item[(i)] $A$ is a $C_0 (X)$-algebra if and only if there exists a continuous map $\phi_A^f : \mathrm{Fac}(A) \rightarrow X$,
\item[(ii)] For $x \in \mathrm{Im} \phi_A$, $I_x = \bigcap \{ I \in \mathrm{Fac}(A) : \phi_A^f (I) = x \}$,
\item[(iii)] For $I \in \Fac (A)$ and $x \in X$, we have $I \supseteq I_x$ if and only if $\phi_A^f (I) = x$.
\end{enumerate}
\end{proposition}
\begin{proof}
(i) If $A$ is a $C_0 (X)$-algebra with base map $\phi_A : \mathrm{Prim}(A) \rightarrow X$ then $\phi_A$ has a unique continuous extension to a map $\phi_A^f : \mathrm{Fac}(A) \rightarrow X$ by Lemma~\ref{l:phifac}.  Conversely, if such a map exists then setting $\phi_A = \restr{\phi_A^f}{\mathrm{Prim}(A)}$ defines a base map.

(ii) For $x \in \mathrm{Im} \phi_A$, we have
\[
I_x = \bigcap \{ P \in \mathrm{Prim}(A) : \phi_A (P) = x \} \supseteq \bigcap \{ I \in \mathrm{Fac}(A) : \phi_A^f (I) = x \}
\]
since $\phi_A^f$ extends $\phi_A$.  Take $I \in \mathrm{Fac}(A)$ such that $\phi_A^f (I) = x$.  Then for all $P \in \mathrm{hull} (I)$, $\phi_A (P) =x$, hence $P \supseteq I_x$ for all such $P$.  Hence $I_x \subseteq k \left( \mathrm{hull}(I) \right) = I$ for all $I \in \mathrm{Fac}(A)$ with $\phi_A^f (I) = x$, so that
\[
I_x \subseteq \bigcap \{ I \in \mathrm{Fac}(A) : \phi_A^f (I) = x \}.
\]
as required.

(iii) It follows from (ii) that if $\phi_A^f (I) = x$ then $I \supseteq I_x$.  Now suppose that $I \supseteq I_x$ and take $P \in \mathrm{hull} (I)$.  Then $P \supseteq I_x$ and so $\phi_A (P) = x$ by~\cite[p. 74]{arch_som_mult}.   It then follows from Lemma~\ref{l:phifac} that
\[
\phi_A^f (I) = \phi_A (P) = x.
\]
\end{proof}

The following lemma identifies when a subalgebra of a $C_0 (X)$-algebra can be identified with a subbundle of the associated upper-semicontinuous C$^{\ast}$-bundle.
\begin{lemma}
\label{l:subcx}
Let $(B,X,\mu_B)$ be a $C_0 (X)$-algebra, and $\iota: A \rightarrow B$ a $\ast$-monomorphism with the property that $\mu_B (f) \iota(a) \in \iota (A)$ for all $a \in A$ and $f \in C_0 (X)$. Then
\begin{enumerate}
\item[(i)] There is a $\ast$-homomorphism $\mu_A : C_0 (X) \rightarrow ZM(A)$ with the property that
\begin{equation}\label{e:iota}
\iota ( \mu_A (f) a ) = \mu_B (f) \iota (a).
\end{equation}
Hence $(A,X, \mu_A )$ is a $C_0 (X)$-algebra, and $\iota$ is a $C_0(X)$-module map, 
\item[(ii)] For $x \in X$ we let $I_x = \mu_A(\set{f \in C_0 (X) : f(x)=0})A$ and $J_x = \mu_B ( \set{f \in C_0 (X) : f(x) = 0} )B$. We have $\iota (I_x) = J_x \cap \iota (A)$,
\item[(iii)] If $B$ is a continuous $C_0 (X)$-algebra then so is $A$.
\end{enumerate}
\end{lemma}
\begin{proof}
(i) Since $\iota ( A) $ is closed under multiplication by $\mu_B (C_0 (X))$,  $\iota (A)$is a closed two-sided ideal of $\iota (A)+ \mu_B (C_0 (X))$, a C$^{\ast}$-subalgebra of $M(B)$.  Then by~\cite[Proposition 3.7(i)]{busby}, there is a $\ast$-homomorphism $\sigma: \iota (A) + \mu_B ( C_0 (X)) \rightarrow M(A)$ extending $\iota^{-1}$ on $\iota (A)$.  Moreover, for $f \in C_0 (X)$ and $a \in A$, we have
\[
(\sigma \circ \mu_B )(f) a = \sigma ( \mu_B (f) \iota (a) ) = \sigma ( \iota (a) \mu_B (f) ) = a (\sigma \circ \mu_B )(f),
\]
since $\sigma \circ \iota$ is the identity on $A$.  As in~\cite[(3.1)]{m_glimm}, we see that $(\sigma \circ \mu_B)(C_0 (X) ) \subseteq ZM(A)$.  Thus we get a $\ast$-homomorphism $\mu_A = \sigma \circ \mu_B : C_0 (X) \rightarrow ZM(A)$.  To see that $\mu_A$ is non-degenerate, let $(f_{\lambda})$ be an approximate identity for $C_0 (X)$.  Then since $\mu_B$ is non-degenerate, $\mu_B (f_{\lambda} )b \rightarrow b$  for all $b \in B$.  In particular, for all $a \in A$ we have 
\[
\mu_A ( f_{\lambda} ) a = \sigma( \mu_B ( f_{\lambda} ) \iota (a) ) \rightarrow \sigma ( \iota (a) ) = a,
\]
which shows that $\mu_A (C_0 (X) )A$ is dense in $A$.  By the Cohen-Hewitt factorisation Theorem, $\mu_A (C_0 (X) )A = A$.  

To see that~(\ref{e:iota}) holds, note that for $f \in C_0 (X)$ and $a \in A$, we have
\begin{eqnarray*}
\iota ( \mu_A (f) a ) & = & \iota \left( (\sigma \circ \mu_B )(f) a \right) \\
& = & \iota \left( \sigma ( \mu_B (f) \iota (a) ) \right) \\
& = & \mu_B (f) \iota (a)
\end{eqnarray*}

Thus $(A, X , \mu_A )$ has the required properties.

(ii)
The inclusion $\iota (I_x) \subseteq J_x \cap \iota (A)$ follows from~(\ref{e:iota}).  Now let $a \in J_x \cap \iota (A)$ and $\varepsilon > 0$.  By upper-semicontinuity there is a neighbourhood $U$ of $x$ in $X$, with compact closure $\overline{U}$, such that $\norm{\iota (a)+J_y} < \varepsilon$ for all $y \in U$.  Choose a continuous $f: X \rightarrow [0,1]$ such that $f(x)=1$ and $f(X \backslash U) \equiv 0$, then $f \in C_0 (X)$ since $\overline{U}$ is compact.  Then 
\[
(1-\mu_A (f) ) a + I_x = (a-f(x)a)+I_x = 0+I_x,
\]
so that $(1-\mu_A (f))a \in I_x$.  

Now we have
\[
\norm{(1-\mu_A (f)) a-a} = \norm{\mu_A (f)a} = \norm{\sigma (\mu_B (f) \iota(a))} = \norm{ \mu_B (f) \iota(a)},
\] 
since  $\sigma$ is injective on $\iota (A)$.  Moreover,
\begin{eqnarray*}
\norm{\mu_B (f) \iota(a)} = \sup_{y \in X} \norm{\mu_B(f) \iota (a) +J_y} &=&\sup_{y\in X} \norm{f(y)( \iota (a)+J_y)} \\
&=& \sup_{y \in U} \abs{f(y)} \norm{\iota (a) + J_y} \\
& \leq & \sup_{y \in U} \norm{\iota (a)+J_y} \leq \varepsilon.
\end{eqnarray*}
Combining these facts, we see that
\[
\norm{a+I_x} \leq \norm{(1-\mu_A (f))a-a} = \norm{\mu_B (f) \iota (a)} \leq \varepsilon.
\]
Since $\varepsilon>0$ was arbitrary and $I_x$ closed, $a \in I_x$.

(iii)    By~\cite[Corollary 1.8.4]{dix} and by part (ii) we may identify
\[
\frac{A}{I_x} = \frac{\iota (A)}{\iota (I_x)} = \frac{\iota (A)}{J_x \cap \iota (A)} \equiv \frac{\iota (A)+J_x}{J_x} \subseteq \frac{B}{J_x}.
\] 
Hence for $a \in A$ and $x \in X$, $\norm{a+I_x} = \norm{\iota (a)+J_x}$.  Since norm functions of elements of $B$ are continuous on $X$, the same is true for $A$.
\end{proof}
\section{Tensor Products of C$^{\ast}$-bundles}
\label{s:tensorbundle}

For a C$^{\ast}$-algebra $A$ we let $\mathrm{Id}'(A)$ be the set of proper closed two-sided ideals of $A$.  For two C$^{\ast}$-algebras $A$ and $B$ let $A \mint B$ be their minimal tensor product.  Define  maps $\Phi, \Delta : \mathrm{Id}' (A) \times \mathrm{Id}' (B) \rightarrow \mathrm{Id}' ( A \otimes_{\alpha} B)$ via
\begin{eqnarray}
\label{e:phi} \Phi (I,J) &=& \mathrm{ker} ( q_I \otimes q_J ) \\
\label{e:delta} \Delta (I,J) &=& I \otimes_{\alpha} B + A \otimes_{\alpha} J,
\end{eqnarray}
where $q_I : A \rightarrow A/ I$ and $q_J : B \rightarrow B/J$ are the quotient $\ast$-homomorphisms.  We remark that by injectivity of the minimal tensor product, $\Delta (I,J)$ is precisely the closure in $A \mint B$ of the kernel of the algebraic $\ast$-homomorphism $q_I \odot q_J : A \odot B \rightarrow (A/I) \odot (B/J)$.

Starting with $I \in \mathrm{Id}' (A \mint B )$, we define $I^A \in \mathrm{Id}'(A)$ and $I^B \in \mathrm{Id}' (B)$ via
\[
I^A= \{ a \in A : a \mint B \subseteq I \}, I^B = \{ b \in B : A \mint b \subseteq I \}.
\]
Defining $\Psi (I) = (I^A,I^B)$ gives a map $\Psi : \mathrm{Id}' (A \mint B ) \rightarrow \mathrm{Id}'(A) \times \mathrm{Id}' (B)$.  Moreover, we have
\[
(\Psi \circ \Phi )(I,J) = ( \Psi \circ \Delta )(I,J) = (I,J)
\]
for all $(I,J) \in \mathrm{Id}'(A) \times \mathrm{Id}' (B)$ by~\cite[Theorem 2.6]{lazar_tensor} and~\cite[Lemma 4.6(ii)]{m_glimm}.
\begin{definition}
Let $A$ and $B$ be C$^{\ast}$-algebras.  Their minimal tensor product $A \mint B$ is said to satisfy Tomiyama's property (F) if $\Delta (I,J) = \Phi (I,J)$ for all $(I,J) \in \mathrm{Id}'(A) \times \mathrm{Id}'(B)$.
\end{definition}
Our definition of property (F) is equivalent to Tomiyama's original definition~\cite{tomiyama}, i.e. that the tensor product states separate the closed ideals of $A \mint B$, by~\cite[Theorem 5]{tomiyama}.  

\begin{definition}
A sequence of $\ast$-homomorphisms between C$^{\ast}$-algebras of the form
\begin{equation}\label{e:exact}
\xymatrix{
0 \ar[r] & J \ar^{\iota}[r] & B \ar^{q}[r] & (B/J) \ar[r] & 0,
}
\end{equation}
where $J$ is an ideal of $B$, $\iota$ the inclusion of $J$ into $B$, and $q$ the quotient $\ast$-homomorphism, is called a \emph{short exact sequence of C$^{\ast}$-algebras}.  We say that a C$^{\ast}$-algebra $A$ is \emph{exact} if the sequence
\begin{equation}
\label{e:exact2}
\xymatrix{
0 \ar[r] & A \mint J \ar^{\mathrm{id} \otimes \iota}[r] & A \mint B \ar^{\mathrm{id} \otimes q}[r] & A \mint (B/J) \ar[r] & 0
}
\end{equation}
is exact for every short exact sequence of the form~(\ref{e:exact}).
\end{definition}

A C$^{\ast}$-algebra $A$ is exact if and only if $A \mint B$ has property (F) for all C$^{\ast}$-algebras $B$.  Clearly if $A$ is such that $A \mint B$ satisfies (F) for all $B$, then for any short exact sequence of the form (\ref{e:exact}), exactness of (\ref{e:exact2}) follows from the fact that $\Delta (\{ 0 \} , J ) = \Phi ( \{ 0 \}, J )$.  The converse is shown in~\cite[Proposition 2.17]{blanch_kirch}.  We will make use of this equivalence repeatedly in the sequel.

The following theorem lists some of the known properties of tensor products of C$^{\ast}$-bundles, and their relation to the maps $\Phi$ and $\Delta$.

\begin{theorem}
\label{t:kirch_wass1}
Let $\mathscr{A} = ( X , A , \pi_x : A \rightarrow A_x )$ and $\mathscr{B} = (Y , B , \sigma_y : B \rightarrow B_y )$ be C$^{\ast}$-bundles over locally compact spaces $X$ and $Y$ respectively. Then
\begin{enumerate}
\item[(i)] the fibrewise tensor product $\mathscr{A} \mint \mathscr{B}$ of $\mathscr{A}$ and $\mathscr{B}$ defined via
\[
\mathscr{A} \mint \mathscr{B} = \left( X \times Y , A \mint B , \pi_x \otimes \sigma_y : A \mint B \rightarrow A_x \mint B_y \right)
\]
is a C$^{\ast}$-bundle over $X \times Y$~\cite[p.136,137]{arch_cb}.
\end{enumerate}
If in addition $\mathscr{A}$ and $\mathscr{B}$ are continuous, then
\begin{enumerate}
\item[(ii)] $\mathscr{A} \mint \mathscr{B}$ is lower-semicontinuous over $X \times Y$~\cite[Proposition 4.9]{kirch_wass}, and
\item[(iii)] for $(x_0,y_0) \in X \times Y$, the function
\[
(x,y) \mapsto \norm{ (\pi_x \otimes \sigma_y )(c)}
\]
is continuous for all $c \in A \mint B$ at $(x_0 , y_0 )$ if and only if
\[
\ker ( \pi_{x_0} \otimes \sigma_{y_0} ) = \ker ( \pi_{x_0} ) \mint B + A \mint \ker ( \sigma_{y_0} ),
\]
that is, if and only if $\Phi ( \ker (\pi_{x_0} ) , \ker ( \sigma_{y_0} ) ) = \Delta ( \ker ( \pi_{x_0} ) , \ker( \sigma_{y_0} ))$~\cite[Theorem 3.3]{arch_cb}.
\end{enumerate}

\end{theorem}

We now introduce an alternative approach to defining a C$^{\ast}$-bundle structure on the tensor product of two (upper-semicontinuous) C$^{\ast}$-bundles, based on the ideal structure of $A \mint B$ rather than the fibrewise tensor product. This construction was considered previously in~\cite{blanch_exact}, in the case where the base spaces are compact, however, we will need additional information on the interplay between the base and structure maps involved.

Suppose that $(A , X , \phi_A )$ is a $C_0(X)$-algebra and $(B,Y, \phi_B )$ a $C_0 (Y) $-algebra, where $\phi_A : \Pm (A) \rightarrow X$ and $\phi_B : \Pm (B) \rightarrow Y$ are their base maps.  Then we get a continuous map $\phi_A \times \phi_B : \Pm (A) \times \Pm (B) \rightarrow X \times Y.$ By a theorem of Lazar~\cite[Theorem 3.2]{lazar_tensor}, we get a unique continuous map $\phi_{\alpha} : \mathrm{Prim}(A \otimes_{\alpha} B ) \rightarrow X \times Y $ such that 
\[
\xymatrix{
\Pm (A) \times \Pm (B) \ar@{^(->}^{\Phi}[r] \ar_{\phi_A \times \phi_B}[rd] & \Pm (A \mint B ) \ar^{\phi_{\alpha}}[d] \\
& X \times Y \\
}
\]
commutes, that is, $\phi_{\alpha} \circ \Phi = \phi_A \times \phi_B$.  Thus, taking $\phi_{\alpha}$ as the base map, $A \mint B$ becomes a $C_0 (X \times Y)$-algebra, $(A \otimes_{\alpha} B , X \times Y , \phi_{\alpha} )$.  

The structure maps $\mu_A : C_0 (X) \rightarrow ZM(A), \mu_B : C_0 (Y) \rightarrow ZM(B)$ and $\mu_{\alpha} : C_0 (X \times Y ) \rightarrow ZM(A \mint B)$ are then uniquely determined by $\phi_A,\phi_B$ and $\phi_{\alpha}$.  We will show in Proposition~\ref{p:tensorbundle} that in fact $\mu_{\alpha}$ may be identified with the map $\mu_A \otimes \mu_B : C_0 (X) \otimes C_0 (Y) \rightarrow ZM(A) \otimes ZM(B) \subseteq ZM(A \mint B)$.

For $x \in X, y \in Y$ define the ideals
\begin{eqnarray*}
I_x & = & \mu_A \left( \{ f \in C_0 (X) : f(x) = 0 \} \right) A \\ &=& \bigcap \{ P \in \mathrm{Prim}(A) : \phi_A (P) = x \} \\
J_y & = & \mu_B \left( \{ g \in C_0 (Y) : g(y) =0 \} \right) B \\ &=& \bigcap \{ Q \in \mathrm{Prim}(A) : \phi_B (Q) = y \} \\
K_{x,y} & = & \mu_{\alpha} \left( \{ h \in C_0 (X \times Y ) : h(x,y) = 0 \} \right) A \otimes_{\alpha} B \\ &=&  \bigcap \{ M \in \mathrm{Prim}(A \otimes_{\alpha} B ) : \phi_{\alpha} (M) = (x,y) \}.\\
\end{eqnarray*}

By~\cite{apt}, there is a canonical injective $\ast$-homomorphism $\iota : M(A) \mint M(B) \rightarrow M(A \mint B )$, and by~\cite[Lemma 3.1]{m_glimm} this map satisfies
\[
\iota ( x \otimes y ) ( a \otimes b ) = (xa) \otimes (yb) \mbox{ and } (a \otimes b)(\iota (x \otimes y ) ) = (ax ) \otimes (by ),
\]
for all elementary tensors $x \otimes y \in M(A) \mint M(B), a \otimes b \in A \mint B$, so that the image $\iota ( ZM(A) \otimes ZM(B) )$ is contained in $ZM(A \mint B )$.  We will suppress mention of $\iota$ and simply consider $ZM(A) \otimes ZM(B) \subseteq ZM(A \mint B )$.
\begin{proposition}
\label{p:tensorbundle}
Suppose $(A, X , \phi_A)$ is a $C_0 (X)$-algebra and $(B , Y , \phi_B )$ a $C_0 (Y)$ algebra.  Then with the above notation: 
\begin{enumerate}
\item[(i)] $A \otimes_{\alpha} B$ is a $C_0 (X \times Y)$-algebra with base map $\phi_{\alpha} : \mathrm{Prim}(A \otimes_{\alpha} B ) \rightarrow X \times Y$ satisfying $\phi_{\alpha} \circ \Phi = \phi_A \times \phi_B$, 
\item[(ii)] The structure map $\mu_{\alpha} : C_0 (X \times Y ) \rightarrow ZM(A \mint B )$ corresponding to $\phi_{\alpha}$ may be identified with $(\mu_A \otimes \mu_B ) \circ \eta$, where $\eta : C_0 (X \times Y ) \rightarrow C_0 (X) \otimes C_0 (Y)$ is the canonical $\ast$-isomorphism and we identify $ZM(A) \otimes ZM(B) \subseteq ZM(A \mint B)$ .
\item[(iii)] Denoting by $\phi_A^f, \phi_B^f$ and $\phi_{\alpha}^f$ the extensions of $\phi_A, \phi_B$ and $\phi_{\alpha}$ to $\mathrm{Fac}(A), \mathrm{Fac}(B)$ and $\mathrm{Fac}(A \otimes_{\alpha} B)$ respectively, we have $\phi_{\alpha}^f \circ \Phi = \phi_A^f \times \phi_B^f$ on $\Fac (A) \times \Fac (B)$,
\item[(iv)] For any $M \in \mathrm{Fac}(A \otimes_{\alpha} B )$ and $(x,y) \in X \times Y$, $\phi_{\alpha}^f (M) = (x,y)$ if and only if $(\phi_A^f \times \phi_B^f )( \Psi (M) ) = (x,y)$,
\item[(v)] For all $(x,y) \in X \times Y$ we have $K_{x,y} = \Delta (I_x , J_y )$.
\end{enumerate}
\end{proposition}
\begin{proof}
(i) is shown in the remarks preceding the proposition.

(ii)  For $f \in C_0 (X \times Y )$, we have $\mu_{\alpha} (f ) = \theta_{\alpha} (f \circ \phi_{\alpha} )$, where $\theta_{\alpha} : C^b ( \Pm (A \mint B) ) \rightarrow ZM(A \mint B)$ is the Dauns-Hofmann isomorphism of equation~(\ref{e:dh}).  For $f \otimes g \in C_0 (X) \otimes C_0 (Y)$ and $(x,y) \in X \times Y$, the $\ast$-isomorphism $\eta$ satisfies $\eta^{-1} ( f \otimes g )(x,y ) = f(x)g(y)$.  Thus by linearity and continuity, it suffices to show that for all $f \otimes g \in C_0 (X) \otimes C_0 (Y)$ and $a \otimes b \in A \mint B$ we have
\[
\mu_{\alpha} ( \eta^{-1} (f \otimes g ) ) (a \otimes b ) = (\mu_A \otimes \mu_B) (f \otimes g )(a \otimes b ) = \mu_A (f) a \otimes \mu_B (g) b.
\]
Take $(P,Q) \in \Pm (A) \times \Pm (B)$, then the Dauns-Hofmann $\ast$-isomorphism $\theta_{\alpha}$ of~(\ref{e:dh}) gives
\begin{eqnarray*}
(\mu_{\alpha} \circ \eta^{-1} ) ( f \otimes g )( a \otimes b) + \Phi (P,Q) & = &  \theta_{\alpha} \left( \eta^{-1} ( f \otimes g ) \circ \phi_{\alpha} \right) (a \otimes b) + \Phi (P,Q) \\
& = & \eta^{-1} ( f \otimes g ) \circ \phi_{\alpha} ( \Phi (P,Q ) ) \left( a \otimes b + \Phi (P,Q) \right) \\
& = & (f \circ \phi_A)(P) ( g \circ \phi_B ) (Q) ( a \otimes b + \Phi (P,Q) ),
\end{eqnarray*}
and since $A \mint B / \Phi (P,Q) \equiv (A/P) \mint (B/Q)$, the last line becomes
\[
\left( (f \circ \phi_A ) (P) ( a + P ) \right) \otimes \left( (g \circ \phi_B )(Q) ( b + Q ) \right).
\]
On the other hand, applying the isomorphisms $\theta_A$ and $\theta_B$ of~(\ref{e:dh}) associated with $A$ and $B$ respectively we get
\begin{eqnarray*}
(\mu_A \otimes \mu_B ) ( f \otimes g )( a \otimes b ) + \Phi (P,Q) & = & \left( \mu_A (f) a + P \right) \otimes \left( \mu_B (g) b + Q \right) \\
& = & \left( \theta_A ( f \circ \phi_A ) a + P \right) \otimes \left( \theta_B ( g \circ \phi_B) b + Q \right) \\
& = & \left( (f \circ \phi_A )(P) (a + P ) \right) \otimes \left( (g \circ \phi_B )(Q)(b+Q) \right).
\end{eqnarray*}
Thus for all $(P,Q) \in \Pm (A) \times \Pm (B)$ we have
\[
\left( (\mu_{\alpha} \circ \eta^{-1} ) ( f \otimes g ) - (\mu_A \otimes \mu_B ) ( f \otimes g ) \right)(a \otimes b) \in \Phi (P,Q),
\]
and since $\bigcap \{ \Phi (P,Q) : (P,Q) \in \Pm (A) \times \Pm (B) \} = \{ 0 \}$, it follows that
\[
(\mu_{\alpha} \circ \eta^{-1} ) ( f \otimes g ) = (\mu_A \otimes \mu_B ) ( f \otimes g ).
\]
(iii) Let $(I,J) \in \mathrm{Fac}(A) \times \mathrm{Fac} (B)$ and $(P,Q) \in \mathrm{hull}(I) \times \mathrm{hull} (J)$.  Then by (i) we have $(\phi_A \times \phi_B) (P,Q) = (\phi_{\alpha} \circ \Phi ) (P,Q)$. By~\cite[Corollary 2.3]{lazar_tensor} we have $\Phi (P,Q) \in \mathrm{hull} ( \Phi (I,J) )$, so that Lemma~\ref{l:phifac} gives
\[
(\phi_A^f \times \phi_B^f )(I,J) = (\phi_A \times \phi_B )(P,Q) = (\phi_{\alpha} \circ \Phi )(P,Q) = (\phi_{\alpha}^f \circ \Phi )(I,J).
\]

(iv) By~\cite[Propositions 2.1 and 4.1]{m_glimm}, $(\Phi \circ \Psi) (M) \in \Fac (A \mint B)$ with $M \supseteq (\Phi \circ \Psi )(M)$, so that $(\Phi \circ \Psi )(M) \in \mathrm{cl} \{ M \}$ and hence
\[
\phi_{\alpha}^f (M) = \phi_{\alpha}^f \left( (\Phi \circ \Psi ) (M) \right).
\]
By (iii), the latter is precisely $(\phi_A^f \times \phi_B^f ) ( \Psi (M) )$.

(v) By (iv) and Proposition~\ref{p:phifac}(iii), we have $M \in \mathrm{hull}_{f} (K_{x,y} )$ if and only if $\Psi (M) \in \mathrm{hull}_f (I_x) \times \mathrm{hull}_f (J_y )$.  But then since $\Psi^{-1} ( \mathrm{hull}_f (I_x) \times \mathrm{hull}_f (J_y ) ) = \mathrm{hull}_f ( \Delta (I_x, J_y ) )$~\cite[Lemma 4.6(i)]{m_glimm}, it follows that $K_{x,y} = \Delta (I_x , J_y )$.
\end{proof}

\begin{definition}
\label{d:tensorbundle}
For a $C_0 (X)$-algebra $(X,A, \mu_A)$ and a $C_0 (Y)$-algebra $(B,Y,\mu_B )$ we will denote by $(A \mint B , X \times Y , \mu_A \otimes \mu_B )$ the $C_0 (X \times Y )$-algebra defined by Proposition~\ref{p:tensorbundle}, and we will consider this construction as the natural (minimal)  tensor product in the category of $C_0 ( - )$-algebras,
\[
(A , X , \mu_A ) \mint ( B , Y , \mu_B ) \equiv (A \mint B , X \times Y , \mu_A \otimes \mu_B ).
\]
\end{definition}

The tensor product construction of definition~\ref{d:tensorbundle} does not agree in general with the fibrewise tensor product bundle studied by Kirchberg and Wassermann in~\cite{kirch_wass}.  This fact may be deduced from~\cite[Lemma 2.3 and Proposition 4.3]{kirch_wass}, and is demonstrated explicitly in~\cite[Proposition 3.1]{blanch_exact}.

We now introduce some properties which characterise when these two notions of the tensor product of a pair of C$^{\ast}$-bundles coincide.  For  C$^{\ast}$-algebras $A$ and $B$ we define the properties $\fgl$ and $\fmp$ on $A \mint B$ as follows:
\begin{align}
\tag{$\mathrm{F}_{\mathrm{Gl}}$} \Phi (G,H) &= \Delta (G,H)  \mbox{ for all }(G,H) \in \Gl (A) \times \Gl (B) \\
\tag{$\mathrm{F}_{\mathrm{MP}}$} \Phi (I,J) &= \Delta (I,J)  \mbox{ for all }(I,J) \in \MP (A) \times \MP (B).
\end{align}
If in addition $(A,X,\mu_A)$ is a $C_0 (X)$-algebra and $(B,Y,\mu_B)$ a $C_0 (Y)$-algebra, we will say that the $C_0 (X \times Y)$-algebra $(A \mint B , X \times Y , \mu_A \otimes \mu_B )$ satisfies \emph{property $\fxy$} if the equation

\begin{equation}\tag{$\mathrm{F}_{X,Y}$}
\Phi (I_x, J_y ) = \Delta (I_x , J_y ) \mbox{ for all } (x,y) \in X \times Y
\end{equation}
holds.  For convenience, we will refer to $\fxy$ as a property of $A \mint B$, rather than $(A \mint B , X \times Y , \mu_A \otimes \mu_B )$, when the context is clear.

We remark that if $A \mint B$ satisfies Tomiyama's property (F) then clearly $A \mint B$ satisfies properties $\fxy$, $\fgl$ and $\fmp$, cf~\cite[Theorem 1.1 and Theorem 2.3]{kaniuth}.  The converse is not true in general; indeed, let $A$ and $B$ be C$^{\ast}$-algebras such that $A \mint B$ does not satisfy property (F), and let $X = \{ x \}$ and $Y = \{ y \}$ be one-point spaces.  Regarding $A$ as a $C_0 (X)$-algebra and $B$ as a $C_0 (Y)$-algebra in the obvious (trivial) way, we have $I_x = \{ 0 \}$ and $J_y = \{ 0 \}$.  Then it is evident that $\Delta (I_x , J_y ) = \Phi (I_x , J_y ) = \{ 0 \}$, so that $A \mint B$ satisfies  property $\fxy$, hence this property does not imply (F).

To see that $\fmp$ and $\fgl$ do not imply (F),  let $A = B = B(H)$, where $H$ is a separable infinite dimensional Hilbert space. Then $\Gl (B(H)) = \MP (B(H)) = \{ 0 \}$, so that as before, $B(H) \mint B(H)$ satisfies $\fmp$ and $\fgl$, but does not satisfy (F) by~\cite{wassermann}.  Other examples are discussed in~\cite[p. 140-141]{arch_cb}.

 The following Theorem relates the $C_0 (X \times Y )$-algebra $(A \mint B , X \times Y , \mu_A \otimes \mu_B )$, its corresponding upper-semicontinuous C$^{\ast}$-bundle, and the fibrewise tensor product of the bundles associated with $A$ and $B$.

\begin{theorem}
\label{c:phi=delta}
Let $(A , X , \mu_A )$ be a $C_0 (X)$-algebra and $(B , Y , \mu_B )$ a $C_0 (Y)$-algebra, and denote by $\mathscr{A} = ( X , A , \pi_x : A \rightarrow A_x )$ and $\mathscr{B} = (Y,B, \sigma_y : B \rightarrow B_y )$ the associated upper-semicontinuous C$^{\ast}$-bundles over $X$ and $Y$ respectively.  Then
\begin{enumerate}
\item[(i)] the $C_0 (X \times Y )$-algebra $(A \mint B , X \times Y , \mu_A \otimes \mu_B )$ defines an upper-semicontinuous C$^{\ast}$-bundle 
\[
\left( X \times Y , A \mint B , \gamma_{(x,y)} : A \mint B \rightarrow (A \mint B )_{(x,y)} \right),
\]
where $(A \mint B )_{(x,y)} = A \mint B / \Delta (I_x , J_y )$ for all $(x,y) \in X \times Y$,
\item[(ii)] the bundle $\left( X \times Y , A \mint B , \gamma_{(x,y)} : A \mint B \rightarrow (A \mint B )_{(x,y)} \right)$ agrees with the fibrewise tensor product bundle $\mathscr{A} \mint \mathscr {B}$ if and only if $A \mint B$ satisfies property $\fxy$,
\item[(iii)] If $\mathscr{A}$ and $\mathscr{B}$ are continuous C$^{\ast}$-bundles and $A \mint B$ satisfies property $\fxy$, then $(A \mint B , X \times Y , \mu_A \otimes \mu_B)$ is a continuous $C_0 (X \times Y )$-algebra.
\end{enumerate}
\end{theorem}
\begin{proof}
(i) is immediate from Proposition~\ref{p:tensorbundle}(v) and the equivalence of $C_0 (X)$-algebras (resp. $C_0 (Y)-, C_0 (X \times Y )-$) and upper-semicontinuous C$^{\ast}$-bundles over $X$ (resp. $Y,X\times Y$).

By definition of the maps $\Phi$ and $\Delta$, for all $(x,y) \in X \times Y$ we have $\Phi (I_x , J_y) = \Delta (I_x , J_y )$ if and only if $( A \mint B )_{(x,y)}  \equiv A_x \mint B_y$, from which (ii) follows.

If $\Phi (I_x , J_y ) = \Delta (I_x , J_y )$ for all $(x,y) \in X \times Y$, then (iii) follows from (ii) and Theorem~\ref{t:kirch_wass1}(iii).
\end{proof}
\begin{remark}
\label{r:tensorbundle}
\begin{enumerate}
\item[(i)]
Given continuous $C_0 (X)$-algebras $(A,X, \mu_A )$ and $(B,Y,\mu_B)$, it is natural to ask whether or not the converse of Theorem~\ref{c:phi=delta}(iii) holds; that is, is property $\fxy$ a necessary condition for the $C_0 (X \times Y)$-algebra $(A \mint B , X \times Y , \mu_A \otimes \mu_B )$ to be continuous.  The analogous result for the fibrewise tensor product is true by Theorem~\ref{t:kirch_wass1}(iii).  We will show in section~\ref{s:ineq} that this is not the case; we can construct such  pairs $(A,X, \mu_A )$ and $(B,Y,\mu_B)$ such that $A \mint B$ does not satisfy property $\fxy$ but $(A \mint B , X \times Y , \mu_A \otimes \mu_B )$ is a continuous $C_0 (X \times Y )$-algebra.  One interesting consequence of this fact is that continuity of the associated fibrewise tensor product is a strictly stronger property than continuity of $(A \mint B , X \times Y , \mu_A \otimes \mu_B )$. 

\item[(ii)]
A special case of Proposition~\ref{p:tensorbundle} arises as follows; let $A$ be a $C_0 (X)$-algebra and $B$ any C$^{\ast}$-algebra.  Then we may regard $B$ as a $C(Y)$-algebra where $Y = \{ y \}$ is a one-point space, so that $X \times Y = X$ and $A \mint B$ is also a $C_0(X)$-algebra.  The base map $\phi_{\alpha} : \Pm (A \mint B) \rightarrow X$ is the extension of $\phi_A \circ p_1 : \Pm(A) \times \Pm (B) \rightarrow X$ to $\Pm(A \mint B)$, where $p_1$ is the projection onto the first factor.  The corresponding structure map is given by $\mu_A \otimes 1 : C_0 (X) \rightarrow ZM(A) \otimes ZM(B) \subseteq ZM(A \mint B )$, where $\mu_A \otimes 1 (f) = \mu_A (f) \otimes 1 $ for all $f \in C_0 (X)$.

Thus by Corollary~\ref{c:phi=delta}(i) we get an upper-semicontinuous C$^{\ast}$-bundle $(X , A \mint B, \gamma_x : A \mint B \rightarrow (A \mint B)_x )$, where $(A \mint B)_x = (A \mint B) / (I_x \mint B )$ for all $x \in X$.  The analogous construction in the fibrewise tensor product case is as follows: for a C$^{\ast}$-bundle $\mathscr{A} = ( X , A , \pi_x : A \rightarrow A_x )$ and a C$^{\ast}$-algebra $B$, we define the C$^{\ast}$-bundle $\mathscr{A} \mint B = ( X , A \mint B , \pi_x \otimes \mathrm{id} : A \mint B \rightarrow A_x \mint B )$.  The two bundles agree precisely when $\Phi (I_x , \{ 0 \} ) = \Delta (I_x , \{ 0 \} )$ for all $x \in X$, by Corollary~\ref{c:phi=delta}(ii).  We will make use of this special case as an intermediate step in the construction of the tensor product of two C$^{\ast}$-bundles in subsequent sections.
\end{enumerate}
\end{remark}

\section{Comparison with the fibrewise tensor product}
\label{s:ineq}

In this section we show that the assumption of property $\fxy$ in Theorem~\ref{c:phi=delta}(iii) is not necessary in general. More precisely, we show that for any inexact C$^{\ast}$-algebra $B$,  there is a continuous $C_0 (X)$-algebra $(A , X, \mu_A )$ such that
\begin{enumerate}
\item[(i)] the fibrewise tensor  product $\mathscr{A} \mint B$, where $\mathscr{A}$ is  the continuous C$^{\ast}$-bundle associated with $(A,X,\mu_A)$,  is discontinuous, while
\item[(ii)] the $C_0 (X )$-algebra $(A \mint B , X  , \mu_A \otimes 1 )$ is continuous.
\end{enumerate}
This shows that the analogue of Archbold's result~\cite[Theorem 3.3]{arch_cb} for the bundles constructed in section~\ref{s:tensorbundle} is untrue.  In particular, we deduce that for continuous $C_0 (X)$-algebras $(A,X,\mu_A)$ and $(B,Y,\mu_B)$, the assumption that the $C_0 (X \times Y )$-algebra $(A \mint B , X \times Y , \mu_A \otimes \mu_B )$ is equal to the fibrewise tensor product (i.e. $A \mint B$ satisfies property $\fxy$), is not a necessary condition for continuity.
\begin{lemma}
\label{l:closed}
Let $(A,X,\mu_A)$ be a $C_0 (X)$-algebra, and denote by $\phi_A^f : \Fac (A) \rightarrow X$ the base map. For any closed subset$F \subseteq X$, setting
\[
I_F = \bigcap \{ I_x : x \in F \},
\]
we have
\begin{enumerate}
\item[(i)]  For $M \in \Fac (A)$  $M \supseteq I_F$ if and only if $\phi_A^f (M) \in F$.
\item[(ii)]  For any C$^{\ast}$-algebra $B$ we have
\[
I_F \mint B = \bigcap_{x \in F} \left( I_x \mint B \right).
\]
\end{enumerate}
\end{lemma}
\begin{proof}
(i) Denote by $m: = \phi_A^f (M)$ and suppose first that $m \not \in F$.  Choose $a \in A$ with $\norm{a + M } = 1$ and $f \in C_0 (X)$ with $f(m) = 1$ and $f(F) = \{ 0 \}$. Then
\[
\mu_A (f) a + M = f ( \phi_A^f (M) ) (a+M ) = f(m)(a + M ) = a +M,
\]
so that $\mu_A(f) a \not\in M$.  On the other hand, for all $x \in F$,
\[
\mu_A(f) a + I_x = f(x) (a + I_x ) = 0,
\]
so that $\mu_A (f) a \in I_F$.  In particular, $M \not\supseteq I_F$.

Now suppose that $m \in F$.  Then by Proposition~\ref{p:phifac}(iii), $M \supseteq I_m $ and so  $M \supseteq \bigcap_{x \in F } I_x = I_F$.

(ii)  We will regard $A \mint B$ as a $C_0 (X)$-algebra as in Remark~\ref{r:tensorbundle}(ii); the base map $\phi_{\alpha} : \Pm (A \mint B) \rightarrow X$ being the unique extension to $\Pm (A \mint B )$ of $\phi_A \circ p_1 : \Pm (A) \times \Pm (B) \rightarrow X$, with $p_1$ the projection onto the first factor.

 Let $M \in \Fac ( A \mint B)$ and let $(M^A,M^B) = \Psi (M)$.  We first show that $M \supseteq I_F \mint B$ if and only if $M^A \supseteq I_F$. By~\cite[Lemma 4.6(ii)]{m_glimm} we have $\Psi (I_F \mint B ) = (I_F , \{ 0 \} )$, and since $\Psi$ is order-preserving, it is clear that if $M \supseteq I_F \mint B$ then $M^A \supseteq I_F$.  On the other hand, since $\Delta$ is also order-preserving, if $M^A \supseteq I_F$ then using~\cite[(4.2)]{m_glimm} we see that
\[
M \supseteq \Delta (M^A , M^B) \supseteq \Delta (I_F, \{ 0 \} ) = I_F \mint B.
\]

By (i) $M^A \supseteq I_F$ if and only if $\phi_A^f (M^A) \in F$.  But then  by Proposition~\ref{p:tensorbundle}(iv), we have
\[ \phi_{\alpha}^f (M) = (\phi_A^f \times \phi_B^f)(M^A,M^B) = (\phi_A^f)(M^A) \]
 and the conclusion follows.

\end{proof}

\begin{lemma}
\label{l:stone}
Let $X$ be an extremally disconnected compact Hausdorff space.  Then any $C(X)$-algebra $(A,X,\mu_A)$ is continuous. 
\end{lemma}
\begin{proof}
For each $x \in X$ we have
\[
\norm{a+I_x} = \inf \{ \norm{(1-\mu_A(f)+f(x))a} f \in C (X) \}
\]
by~\cite[Lemme 1.10]{blanch_def}.  Moreover, it is easily seen that for a given $f \in C (X)$ and $a \in A$, the norm function $x \mapsto \norm{(1-\mu_A(f)+f(x))a}$ is continuous on $X$.  Since $X$ is extremally disconnected and compact, $C(X)$ is monotone complete, and so the above infinum belongs to $C(X)$.
\end{proof}

\begin{proposition}
\label{p:vna}
Let $B$ be a C$^{\ast}$-algebra, $M$ a von Neumann algebra and $(M , \Gl (M) , \theta_M )$ the $C(\Gl(M))$-algebra associated with the Dauns-Hofmann representation of $M$.  Then $(M \mint B , \Gl (M) , \theta_M \otimes 1 )$ is a continuous $C( \Gl (M) )$-algebra.
\end{proposition}
\begin{proof}
Since $Z(M)$ is a von Neumann algebra, $\Gl (M) = \mathrm{Prim} (Z(M))$ is  an extremally disconnected compact Hausdorff space.  Continuity of $(M \mint B , \Gl (M) , \theta_M \otimes 1 )$ then follows from Lemma~\ref{l:stone}.
\end{proof}

\begin{theorem}
\label{t:ineq}
Let $B$ be an inexact C$^{\ast}$-algebra.  Then there is a von Neumann algebra $M$, whose Dauns-Hofmann representation $(M , \Gl (M) , \theta_M )$ and associated continuous C$^{\ast}$-bundle $(\Gl (M) ,M , \pi_p:M \rightarrow M_p )$ satisfy
\begin{enumerate}
\item[(i)] $(M \mint B , \Gl (M) , \theta_M \otimes 1 )$ is a continuous $C(\Gl (M) )$-algebra, and
\item[(ii)] $(M \mint B , \Gl (M) , \pi_p \otimes \mathrm{id}_B : M \mint B \rightarrow M_p \mint B )$ is a discontinuous C$^{\ast}$-bundle.
\end{enumerate}
If in addition $B$ is a prime C$^{\ast}$-algebra (e.g. if $B$ primitive), then $M \mint B$ is quasi-standard.
\end{theorem}
\begin{proof}
Let $M = \prod_{n=1}^{\infty} M_n ( \mathbb{C} )$.  Then $M$ is a von Neumann algebra, hence it is quasi-standard by~\cite[Section 5]{arch_primal}.  Moreover, $Z(M)$ consists of the sequences $(\lambda_n 1_n)_{n=1}^{\infty} \in M$, where $\lambda_n \in \mathbb{C}$ and $1_n $ is the $n \times n$ identity matrix.   It follows that $\Gl (M) = \Pm (Z(M))$ is canonically homeomorphic to $\beta \mathbb{N}$.

Since $B$ is inexact, the sequence
\begin{equation}
\tag{$\dagger$}
0 \rightarrow I_0 \mint B \rightarrow M \mint B \rightarrow (M / I_0) \mint B \rightarrow 0
\end{equation}
is inexact by~\cite{kirch_fubini}.  We claim that there is some $q \in \beta \mathbb{N}$ for which $G_q \mint B \subsetneq \ker ( \pi_q \otimes \mathrm{id}_B )$.  Suppose not. By~\cite[Lemma 2.2]{lazar_tensor}
\[
\ker ( \pi_0 \otimes \mathrm{id}_B ) = \bigcap_{q \in \beta \mathbb{N} \backslash \mathbb{N}} \ker ( \pi_q \otimes \mathrm{id}_B ) ,
\] 
and by Lemma~\ref{l:closed}
\[
I_0 \mint B = \bigcap_{q \in \beta \mathbb{N} \backslash \mathbb{N}} I_q \mint B. 
\]
Thus if $I_q \mint B = \ker ( \pi_q \otimes \mathrm{id}_B )$ for all $q \in \beta \mathbb{N} \backslash \mathbb{N}$, the above intersections would agree, which would imply that $(\dagger)$ were exact, which is not the case.

Thus  $(M \mint B , \beta \mathbb{N} , \pi_q \otimes \mathrm{id}_B : M \mint B \rightarrow M_q \mint B )$ is discontinuous at some point $p \in \beta \mathbb{N} \backslash \mathbb{N}$ by~\cite[Proposition 2.7]{kirch_wass}.

On the other hand, by Proposition~\ref{p:vna}, the $C( \Gl(M) )$-algebra $(M \mint B , \Gl (M) , \theta_M \otimes 1)$ is continuous.

Under the additional assumption that the zero ideal of $B$ is prime, then necessarily we have $\Gl (B) = \{ 0 \}$ by~\cite[Lemma 2.2]{arch_som_qs}, and so $\Gl (M \mint B )$ is homeomorphic to $\beta \mathbb{N}$ in the obvious way.  In particular, $(M \mint B , \Gl (M) , \theta_M \otimes 1 )$ corresponds to the Dauns-Hofmann representation of $M \mint B$, and the fibre algebras are prime throughout a dense subset of $\Gl (M \mint B)$, namely the points of $\mathbb{N}$.  By~\cite[Theorem 3.4 (iii)$\Rightarrow$(i)]{arch_som_qs}, $M \mint B$ is quasi-standard.

\end{proof}

\section{Continuity and exactness of the $C_0 (X \times Y)$-algebra $A\mint B$}

In this section we investigate the relationship between exactness of a continuous $C_0 (X)$-algebra $(A , X , \mu_A )$ and its minimal tensor product with an arbitrary continuous $C_0 (Y)$-algebra $(B,Y,\mu_B)$. The corresponding result regarding continuity of fibrewise tensor products of C$^{\ast}$-bundles was obtained by Kirchberg and Wassermann in~\cite{kirch_wass}:

\begin{theorem}\emph{(E. Kirchberg, S. Wassermann~\cite[Theorem 4.5]{kirch_wass})}
The following conditions on a C$^{\ast}$-algebra $B$ are equivalent:
\begin{enumerate}
\item[(i)] $B$ is exact,
\item[(ii)] For every locally compact Hausdorff space $X$ and continuous C$^{\ast}$-bundle $\mathscr{A} = ( X , A , \pi_x : A \rightarrow A_x )$ over $X$, the fibrewise tensor product $\mathscr{A} \mint B$ is continuous,
\item[(iii)] For every separable, unital continuous C$^{\ast}$-bundle $\mathscr{A} = ( \hat{\mathbb{N}} , A , \pi_n : A \rightarrow A_n )$ over $\hat{\mathbb{N}}$, the fibrewise tensor product $\mathscr{A} \mint B$ is continuous.
\end{enumerate}
\end{theorem}
While Theorem~\ref{t:ineq} shows that continuity of a $C_0 (X \times Y)$-algebra of the form $(A \mint B , X \times Y , \mu_A \otimes \mu_B )$ is a  strictly weaker property than continuity of the corresponding fibrewise tensor product, a discontinuous example was already exhibited by Blanchard in~\cite{blanch_exact}.  The construction of this counterexample depends heavily on the specific properties of the algebras involved.  Our main result of this section, Theorem~\ref{t:disc}, shows that this pathology is in some sense universal; more precisely, we show that $(A,X,\mu_A)$ is exact if and only if the  $C_0 (X \times Y)$-algebra $(A \mint B , X \times Y , \mu_A \otimes \mu_B )$ is continuous for each continuous $C_0 (Y)$-algebra $(B,Y,\mu_B)$. 

The following two lemmas are known, we include a proof for completness.
\begin{lemma}
\label{l:cstnorm}
Let $A$ and $B$ be C$^{\ast}$-algebras and $(I,J) \in \mathrm{Id}'(A) \times \mathrm{Id}' (B)$.  Then the quotient C$^{\ast}$-algebra $A \mint B / \Delta (I,J)$ is naturally isomorphic to $(A/I) \otimes_{\gamma} (B/J)$, where $\norm{\cdot}_{\gamma}$ is a C$^{\ast}$-norm on $(A/I) \odot (B/J)$.  Moreover, $\norm{\cdot}_{\gamma} = \norm{\cdot}_{\alpha}$ if and only if $\Phi(I,J) = \Delta (I,J)$.
\end{lemma}
\begin{proof}

Let $\pi_I: A \rightarrow A/I$ and $\pi_J : B \rightarrow B/J$ be the quotient maps.   We remark that if $(\pi_I \odot \pi_J ) : A \odot B \rightarrow (A/I) \odot (B/J)$ denotes the canonical algebraic $\ast$-homomorphism, then the closure of its kernel in $A \mint B$ is $\overline{\ker (\pi_I \odot \pi_J )} = \Delta(I,J)$. 

Take $z=\sum_{i=1}^n x_i \otimes y_i \in (A/I) \odot (B/J)$ and choose $a_1 ,\ldots , a_n \in A$ and $b_1 , \ldots , b_n \in B$ such that $\pi_I (a_i) = x_i$ and $\pi_J (b_i) = y_i$ for $1 \leq i \leq n$ and set $c = \sum_{i=1}^n a_i \otimes b_i$, so that $(\pi_I \odot \pi_J )(c) = z$.  Define $\gamma: (A/I) \odot (B/J) \rightarrow [0, \infty)$ via $\gamma (z) = \norm{c + \Delta (I,J)}$.  Then $\gamma$ is well-defined since if $c' \in A \odot B$ also satisfies $(\pi_I \odot \pi_J)(c')=z$, then $c-c' \in \ker ( \pi_I \odot \pi_J ) \subseteq \Delta(I,J)$, hence
\[
\gamma (c') = \gamma (c'-c+c) = \norm{c'-c+c + \Delta(I,J)} = \norm{c+\Delta(I,J)} = \gamma (c).
\]
Clearly $\gamma$ is a seminorm, and since $z^{\ast}z = (\pi_I \odot \pi_J)(c^{\ast}c)$, we have
\[
\gamma(z^{\ast}z) = \norm{c^{\ast}c+\Delta(I,J)} = \norm{(c+\Delta(I,J))^*(c+\Delta(I,J))} = \norm{c+\Delta (I,J)}^2 = \gamma (z)^2,
\]
(by the C$^{\ast}$-condition on the quotient norm), so $\gamma$ is a C$^{\ast}$-seminorm.  Finally, if $\gamma(z)=0$ then $c \in \Delta(I,J) \cap A \odot B = \ker (\pi_I \odot \pi_J)$, so that $z=0$ in $(A/I) \odot (B/J)$.  It follows that $\gamma$ is a well-defined C$^{\ast}$-norm on $(A/I) \odot (B/J)$.

It follows that $\pi_I \odot \pi_J : A \odot_{\alpha} B \rightarrow (A/I) \odot_{\gamma} (B/J)$ is a bounded, surjective $\ast$-homomorphism of normed $\ast$-algebras, and hence extends to a $\ast$-homomorphism $\pi_I \otimes_{\gamma} \pi_J : A \mint B \rightarrow (A/I) \otimes_{\gamma} (B/J)$.  Since the range of $\pi_I \otimes_{\gamma} \pi_J$ is closed and contains the dense set  $(A/I) \odot (B/J)$, it is surjective.  We claim that $\ker ( \pi_I \otimes_{\gamma} \pi_J) = \Delta (I,J)$.

Since $\ker(\pi_I \otimes_{\gamma} \pi_J)$ is closed and contains $I \odot B + A \odot J$, it must also contain $\Delta(I,J)$.  To show the reverse inclusion, let $d \in \ker ( \pi_I \otimes_{\gamma} \pi_J )$ and $\varepsilon >0$.  Then there is $c \in A \odot B$ with $\norm{c-d}< \frac{\varepsilon}{2}$.  By the definition of $\norm{\cdot }_{\gamma}$, 
\[ \norm{ c + \Delta(I,J)} = \norm{(\pi_I \otimes_{\gamma} \pi_J)(c)} = \norm{(\pi_I \otimes_{\gamma} \pi_J)(c-d)} < \frac{\varepsilon}{2}  \]
(since $d \in \ker ( \pi_I \otimes_{\gamma} \pi_J)$).  It follows that
\[
\norm{d+ \Delta (I,J)} = \norm{d-c+c+\Delta (I,J)} \leq \norm{d-c} + \norm{c+ \Delta (I,J)} < \frac{\varepsilon}{2} + \frac{\varepsilon}{2} = \varepsilon.
\]
Since $\varepsilon$ was arbitrary, $d \in \Delta(I,J)$.  We have shown that $\ker( \pi_I \otimes_{\gamma} \pi_J) = \Delta(I,J)$, and hence we can conclude that $(A \mint B)/ \Delta(I,J)$ is canonically $\ast$-isomorphic to $(A/I) \otimes_{\gamma}(B/J)$.

For the final assertion, note that  $\norm{\cdot}_{\gamma} = \norm{\cdot}_{\alpha}$ if and only if $\pi_I \otimes_{\gamma} \pi_J$ is  the canonical $\ast$-homomorphism $\pi_I \otimes \pi_J :A \mint B \rightarrow (A/I) \mint (B/J)$, whose kernel is by definition $\Phi(I,J)$.  Hence the two norms are equal if and only if $\ker (\pi_I \otimes_{\gamma} \pi_J)$ (=$\Delta(I,J)$) is equal to $\Phi(I,J)$.   
\end{proof}

\begin{lemma}
\label{l:nbhd}
Let $\mathcal{A} = (X , A , \pi_x : A \rightarrow A_x )$ be a C$^{\ast}$-bundle and $x_0 \in X$.  Then for each $a \in A$
\[
\| a + I_{x_0} \| = \inf_W \sup_{x \in W} \| \pi_x (a) \|,
\]
as $W$ ranges over all open neighbourhoods of $x_0 $ in $ X$.
\end{lemma}
\begin{proof}
Fix an open neighbourhood $U$ of $x_0$ in $X$.  We first claim that
\[
\sup_{x \in U } \| a + I_x \| = \sup_{x \in U} \| \pi_x (a) \|.
\]
It is clear that $\ker ( \pi_x ) \supseteq I_x$, hence we have $\| \pi_x (a) \| \leq \| a + I_x \|$ for all $x \in U$ and so the supremum on the left is always greater than or equal to that on the right.

Let $x_1 \in U$ and choose $f \in C_0 (X)$, $0\leq f \leq 1$, with $f (x_1) =1$ and $f(X \backslash U ) \equiv 0 $, then $ \| a + I_{x_1} \| = \| f \cdot a + I_{x_1} \|$.  Moreover,
\[
\| f \cdot a \|  = \sup_{x \in X } \| \pi_x (f \cdot a ) \| = \sup_{x \in U} \| \pi_x ( f \cdot a ) \| = \sup_{x \in U} | f (x) | \cdot \| \pi_x (a) \| \leq \sup_{x \in U } \| \pi_x (a) \|,
\]
whence
\[
\| a + I_{x_1} \| = \| f \cdot a + I_{x_1} \| \leq \| f \cdot a \| \leq \sup_{x \in U } \| \pi_x (a) \|.
\]
It follows that
\[
\sup_{x \in U } \| a + I_{x} \| \leq \sup_{x \in U } \| \pi_x (a) \|,
\]
and so 
\[
\sup_{x \in U } \| a + I_{x} \| = \sup_{x \in U } \| \pi_x (a) \|.
\]

Suppose for a contradiction that
\[
\alpha : = \inf_W \sup_{x \in W } \| \pi_x (a) \| > \| a + I_{x_0} \|.
\]
Then since $x \mapsto \| a + I_x \|$ is upper semicontinuous on $X$, we could find an open neighborhood $U$ of $x_0$ such that
\[
\| a + I_x \| < \left( \frac{\alpha+ \| a + I_{x_0} \|}{2}  \right) \mbox{ for all } x \in U.
\]
But this would then imply that 
\[
\sup_{x \in U} \| a + I_x \|  \leq  \left( \frac{\alpha+ \| a + I_{x_0} \| }{2}  \right) < \alpha \leq \sup_{x \in U} \| \pi_x (a) \|,
\]
contradicting the fact that these suprema must be equal for all open neighbourhoods $U$ of $x_0$.
\end{proof}

\begin{proposition}
\label{p:disc}
Let $B$ be an inexact C$^{\ast}$-algebra.  Then there is a separable unital $C( \nhat )$-algebra $(  A, \nhat ,  \mu_A )$ such that the $C( \nhat )$-algebra $( A \mint B , \nhat ,  \mu_A \otimes 1 )$ is discontinuous at $\infty$.
\end{proposition}
\begin{proof}
Since $B$ is inexact, by~\cite[Proposition 4.2]{kirch_wass} there is a separable, unital continuous C$^{\ast}$-bundle $\mathcal{C} = (  C , \nhat ,  \sigma_n : C \rightarrow C_n )$ with the property that the minimal fibrewise tensor product $\mathcal{C} \mint B = ( C \mint B , \nhat, \sigma_n \otimes \mathrm{id} : C \mint B \rightarrow C_n \mint B )$ is discontinuous at $\infty$.  

Since $C$ is continuous, for each $n \in \nhat $ we have $I_n  = \ker \sigma_n$.  By~\cite[Theorem 3.3]{arch_cb}, we have $I_{\infty} \mint B \subsetneq \ker ( \sigma_{\infty} \otimes \mathrm{id} )$, while $I_n \mint B = \ker ( \sigma_n \otimes \mathrm{id} )$ for all $n \in \mathbb{N}$.  It follows in particular that $(C \mint B )/(I_{\infty} \mint B )$ is canonically $\ast$-isomorphic to $C_{\infty} \otimes_{\gamma} B$, where $\norm{\cdot}_{\gamma}$ is a C$^{\ast}$-norm on $C_{\infty} \odot B$ distinct from $\norm{\cdot}_{\alpha}$ by Lemma~\ref{l:cstnorm}.

There is thus $y = \sum_{i=1}^{\ell} c_{\infty}^{(i)} \otimes b^{(i)} \in C_{\infty} \odot B$ with the property that $\norm{y}_{\gamma} > \norm{y}_{\alpha}$. Choose $c^{(1)}, \ldots ,c^{(\ell)} \in C$ such that $\sigma_{\infty}(c^{(i)}) = c_{\infty}^{(i)}$ for $1 \leq i \leq \ell$, and  set $\overline{y} = \sum_{i=1}^{\ell} c^{(i)} \otimes b^{(i)} \in C \odot B$.  Then 
\[ \norm{\overline{y} + I_{\infty} \mint B } = \norm{y}_{\gamma} > \norm{y}_{\alpha} = \norm{(\sigma_{\infty} \otimes \mathrm{id})(y)}, \]
by the definitions of $\norm{\cdot}_{\gamma}$ and $\sigma_{\infty} \otimes \mathrm{id}$.

Now let $D = C( \nhat ) \mint C_{\infty} $ be the trivial (hence continuous) bundle on $\nhat$ with fibre $C_{\infty}$, and denote by $\varepsilon_n$ the evaluation maps, where $n \in \nhat$.  Let $A$ be the pullback C$^{\ast}$-algebra in the diagram
\[
\xymatrix{ A \ar[r] \ar[d] & C  \ar^{\sigma_{\infty}}[d] \\
D  \ar_{\varepsilon_{\infty}}[r] & C_{\infty}  \\
}
\]
that is, the C$^{\ast}$-subalgebra of $C \oplus D$ given by
\[
A = \{ c \oplus d \in C \oplus D : \sigma_{\infty}(c) = \varepsilon_{\infty} (d) \}.
\]
Then we have a well-defined $\ast$-homomorphism $\pi_{\infty}: A \rightarrow C_{\infty} $ sending $c \oplus d$ to $\sigma_{\infty} (c) = \varepsilon_{\infty} (d)$~\cite[2.2]{pedersen_pull}.  

For $n \in \mathbb{N}$ we may extend (keeping the same notation) $\sigma_n$ and $\varepsilon_n$ to $A$ by setting
\[
\sigma_n ( c \oplus d ) = \sigma_n (c) , \varepsilon_n ( c \oplus d ) = \varepsilon_n (d).
\]
$A$ defines a continuous C$^{\ast}$-bundle $( A , \nhat ,  \pi_n : A \rightarrow A_n )$ as follows: for $n \in \mathbb{N}$ set $A_{2n-1} = C_n$ and $\pi_{2n-1} = \sigma_n$, $A_{2n} = C_{\infty}$ and $\pi_{2n} = \varepsilon_n$, and $\pi_{\infty}: A \rightarrow A_{\infty} = C_{\infty}$ as above.  Continuity of $A$ follows easily from that of $C$ and $D$:  for $c \oplus d \in A$ we have $\norm{\sigma_n (c) }$ and $\norm{\varepsilon_n (d)}$ both converge to $\norm{\pi_{\infty}(c \oplus d)}$, it follows that $\norm{\pi_n (c \oplus d)} \rightarrow \norm{\pi_{\infty}(c \oplus d)}$.

Regarding $A$ as a $C( \nhat )$-algebra, let $\mu_A: C( \nhat ) \rightarrow Z(A)$ be the structure map; where $f \in C( \nhat )$ acts by pointwise multiplication.  Now $A \otimes B$ is a $C( \nhat )$-algebra with structure map $\mu_A \otimes 1 : C( \nhat ) \otimes \mathbb{C} \rightarrow ZM(A \mint B)$.  For each $n \in \nhat$ we let $K_n$ be the ideal
\[
(\mu \otimes 1) \left( \{ f \in C( \nhat ) : f(n) = 0 \} \otimes \mathbb{C} \right) A \mint B
\]
of $A \mint B$, so that $n \mapsto \norm{ x + K_n }$ is upper-semicontinuous on $\nhat$.  Again it follows from~\cite[Theorem 3.3]{arch_cb} that for $n \in \mathbb{N}$  we have $K_n = \ker ( \pi_n \otimes \mathrm{id} )$, and by Proposition~\ref{p:tensorbundle}(v) we have $K_{\infty} = I_{\infty} \mint B$.

On the other hand the (lower-semicontinuous) C$^{\ast}$-bundle $\mathcal{A} \mint B = ( \nhat , A \mint B , \pi_n \otimes \mathrm{id} : A \mint B \rightarrow A_n \mint B ) $ has fibres $A_n \mint B =  C_{\frac{n+1}{2}} \mint B$ for $n$ odd, $A_n \mint B = C_{\infty} \mint B$ for $n$ even, and $A_{\infty} = C_{\infty} \mint B$.

 For $1 \leq i \leq \ell$ let $d^{(i)} \in D$ be the constant section $\varepsilon_n ( d^{(i)} ) = c^{(i)}_{\infty}$,  set $a^{(i)} = c^{(i)} \oplus d^{(i)} \in A$, and let $x = \sum_{i=1}^{\ell} a^{(i)} \otimes b^{(i)}.$  Then for $n \in \mathbb{N}$ we have
\begin{eqnarray*}
(\pi_n \otimes \mathrm{id} ) (x) & = & (\pi_n \otimes \mathrm{id} )\left( \sum_{i=1}^{\ell} a^{(i)} \otimes b^{(i)} \right) \\
& = & \sum_{i=1}^{\ell} \pi_n (a^{(i)} ) \otimes b^{(i)} \\
& = & \displaystyle \left\{ \begin{array}{rcl}
\sum_{i=1}^{\ell} \sigma_{\frac{n+1}{2}} (c^{(i)} ) \otimes b^{(i)} & \mbox{ if } & n \mbox{ odd} \\
\vspace{5pt} & & \\
\sum_{i=1}^{\ell} \varepsilon_{\frac{n}{2}} (d^{(i)} ) \otimes b^{(i)} & \mbox{ if } & n \mbox{ even}\\
\end{array} \right. \\
& = & \left\{ \begin{array}{rcl}
(\sigma_{\frac{n+1}{2}} \otimes \mathrm{id} )( \overline{y} ) & \mbox{ if } & n \mbox{ odd} \\
\vspace{5pt} & & \\
\sum_{i=1}^{\ell} c_{\infty}^{(i)} \otimes b^{(i)} & \mbox{ if } & n \mbox{ even} \\
\end{array} \right.
\\
\end{eqnarray*} 
 
It then follows that
\[
\norm{ ( \pi_n \otimes \mathrm{id} )(x)}  =  \left\{ \begin{array}{rcl}
\norm{ (\sigma_{\frac{n+1}{2}} \otimes \mathrm{id} )( \overline{y} )} & \mbox{ if } & n \mbox{ odd} \\
\vspace{5pt} & & \\
\norm{ \sum_{i=1}^{\ell} c_{\infty}^{(i)} \otimes b^{(i)} }_{C_{\infty} \mint B} = \norm{y}_{\alpha} & \mbox{ if } & n \mbox{ even}.
\end{array}
\right.
\]

Finally, note that by Lemma~\ref{l:nbhd} we have
\begin{eqnarray*}
\norm{x+K_{\infty}} & = & \max \left( \limsup_n \norm{(\pi_n \otimes \mathrm{id})(x)} , \norm{(\pi_{\infty} \otimes \mathrm{id})(x)} \right) \\
\norm{y}_{\gamma}  = \norm{ \overline{y} + I_{\infty} \mint B } & = & \max \left( \limsup_n \norm{ ( \sigma_n \otimes \mathrm{id} )(\overline{y})} , \norm{(\sigma_{\infty} \otimes \mathrm{id} )(\overline{y})} \right). \\
\end{eqnarray*}
Since we know that $\norm{\sigma_{\infty} \otimes \mathrm{id} )(\overline{y})} = \norm{y}_{\alpha} < \norm{y}_{\gamma}$, the second equality becomes
\[
\norm{y}_{\gamma} = \limsup_n \norm{(\sigma_n \otimes \mathrm{id} )(\overline{y})}.
\]
In particular, we conclude that
\begin{eqnarray*}
\norm{x + K_{\infty}} & \geq & \limsup_n \norm{ (\pi_n \otimes \mathrm{id} )(x) } \\
& \geq & \limsup \{ \norm{ (\pi_n \otimes \mathrm{id} )(x) } : n \mbox{ odd} \}\\
& = & \limsup \{ \norm{ ( \sigma_{\frac{n+1}{2}} \otimes \mathrm{id}) ( \overline{y} )} : n \mbox{ odd } \} \\
& = & \norm{y}_{\gamma}. 
\end{eqnarray*}

Thus for all even $n$  (since $K_n = \ker( \pi_n \otimes \mathrm{id} )$ for $n \in \mathbb{N}$ by continuity at these points) \[ \norm{x + K_n } = \norm{ ( \pi_n \otimes \mathrm{id} )(x) } = \norm{y}_{\alpha} < \norm{y}_{\gamma} \leq \norm{x + K_{\infty}}, \] and it follows that $n \mapsto \norm{x + K_n}$ is discontinuous at $\infty$.

\end{proof}

\begin{corollary}
\label{c:disc}
Let $B$ be an inexact C$^{\ast}$-algebra.  Then there is a  separable unital, continuous $C(\nhat )$-algebra $(\tilde{A} , \nhat , \mu_{\tilde{A}} )$, with $\Pm (\tilde{A})$ homeomorphic to $\nhat$, such that the $C( \nhat )$-algebra $(\tilde{A} \mint B , \nhat , \mu_{\tilde{A}} \otimes 1 )$ is discontinuous at $\infty$.  Moreover, $\mathrm{Prim} ( \tilde{A} )$ is canonically homeomorphic to $\nhat$, and $\mu_{\tilde{A}}$ agrees with the Dauns-Hofmann $\ast$-isomorphism $\theta_{\tilde{A}} : C( \Pm (\tilde{A}) ) \rightarrow Z(\tilde{A})$.
\end{corollary}
\begin{proof}
Let $(A , \nhat , \mu_A )$ be the separable, unital continuous $C( \nhat )$-algebra of Proposition~\ref{p:disc}. Then by~\cite[Corollary 4.7]{blanch_free} there is a  unital continuous $C(\nhat )$-algebra $(\tilde{A} , \nhat , \mu_{\tilde{A}} )$ with simple fibres and a $C(\nhat)$-module $\ast$-monomorphism $\iota :A \rightarrow \tilde{A}$.  By injectivity of the minimal tensor product we get a $\ast$-monomorphism $\iota \otimes \mathrm{id} : A \mint B \rightarrow \tilde{A} \mint B$.

Now, $\tilde{A} \mint B$ is a $C(\nhat)$-algebra with base map $\mu_{\tilde{A}} \otimes 1$.  To show that $ \iota \otimes \mathrm{id}$ is a $C( \nhat )$-module map, take $a \otimes b \in A \mint B$ and $f \in C( \nhat )$.  Then we have
\begin{eqnarray*}
(\mu_{\tilde{A}} \otimes 1 )(f)(\iota \otimes \mathrm{id} ) (a \otimes b ) & = & (\mu_{\tilde{A}} (f) \iota (a) ) \otimes b \\
& = & (\iota \otimes \mathrm{id} ) ( \mu_A (f)a \otimes b ) \in (\iota \otimes \mathrm{id} )(A \mint B ).
\end{eqnarray*}
In particular it follows from Lemma~\ref{l:subcx}(iii) that $(\tilde{A} \mint B , \nhat , \mu_{\tilde{A}} \otimes 1 )$ is discontinuous, since it contains the discontinuous $C(\nhat)$-algebra $(A \mint B , \nhat , \mu_A \otimes 1 )$ as a $C(\nhat)$-submodule.

Denote by $\phi_{\tilde{A}} : \Pm (A) \rightarrow \nhat $ the base map uniquely determined by $\mu_{\tilde{A}}$, and by
\[
\tilde{I}_n = \mu_{\tilde{A}} \left(\{  f \in C( \nhat ): f(n) = 0 \} \right) \tilde{A} = \bigcap \{ P \in \Pm ( \tilde{A} ) : \phi_{\tilde{A}} (P) = n \}
\]
for each $n \in \nhat$.  Then since each fibre $\tilde{A} / \tilde{I}_n$ is simple, it follows that $\tilde{I}_n$ is maximal (and in particular primitive) for all $n \in \nhat$.  Moreover, since every $P \in \Pm ( \tilde{A} )$ contains a unique $\tilde{I}_n$ for some $n \in \nhat$, we see that $\tilde{I}_n \mapsto n$ is a bijection.  The fact that this map is a homeomorphism then follows from Lee's theorem~\cite[Theorem 4]{lee}.
\end{proof}

\begin{theorem}
\label{t:disc}
The following conditions on a C$^{\ast}$-algebra $B$ are equivalent:
\begin{enumerate}
\item[(i)] $B$ is exact,
\item[(ii)] for every separable, unital continuous $C( \hat{\mathbb{N}} )$-algebra $(A, \hat{\mathbb{N}}, \mu_A )$, the $C( \hat{\mathbb{N}} )$-algebra $(A \mint B , \hat{\mathbb{N}} , \mu_A \otimes 1 )$ is continuous,
\item[(iii)] for every separable, unital C$^{\ast}$-algebra $\tilde{A}$ with $\Pm (\tilde{A})$ Hausdorff, the $C(\Pm(\tilde{A}))$-algebra $(\tilde{A} \mint B , \Pm (\tilde{A}) , \theta_{\tilde{A}} \otimes 1)$  is continuous, where $\theta_{\tilde{A}}: C( \Pm (\tilde{A})) \rightarrow Z(\tilde{A})$ is the Dauns-Hofmann $\ast$-isomorphism,
\item[(iv)] for every locally compact Hausdorff space $X$ and continuous $C_0 (X)$-algebra $(A , X , \mu_A )$, the $C_0 (X)$-algebra $(A \mint B , X , \mu_A \otimes 1 )$ is continuous.
\end{enumerate}
\end{theorem}
\begin{proof}
(i)$\Rightarrow$(iv): Suppose that $B$ is exact and let $(A , X , \mu_A )$ be a continuous $C_0 (X)$-algebra.  Then since $A \mint B$ has property (F), we have $\Delta (I_x , \{ 0 \} ) = \Phi (I_x , \{ 0 \} )$ for all $x \in X$.  It then follows from Corollary~\ref{c:phi=delta}(iii) that $(A \mint B , X , \mu_A \otimes 1 )$ is continuous.

(iv)$\Rightarrow$(iii) and (iv)$\Rightarrow$(ii) are evident.   

(ii)$\Rightarrow$(i): Suppose that $B$ is inexact, then by Proposition~\ref{p:disc} there is a separable, unital continuous $C( \hat{\mathbb{N}} )$-algebra $(A , \hat{\mathbb{N}} , \mu_A )$ such that $(A \mint B, \hat{\mathbb{N} }, \mu_A \otimes 1 )$ is discontinuous, so that (ii) fails.

(iii)$\Rightarrow$(ii): This follows similarly from Corollary~\ref{c:disc}.

\end{proof}

\begin{corollary}
\label{c:disc2}
The following conditions on a $C_0 (Y)$-algebra $(B,Y,\mu_B)$ are equivalent:
\begin{enumerate}
\item[(i)] $B$ is exact,
\item[(ii)] For every separable unital $C_0( X )$-algebra $(A,X, \mu_A)$, with $X = \nhat$ , the $C_0 ( X \times Y)$-algebra $(A \mint B , X \times Y , \mu_A \otimes \mu_B )$ satisfies property $\fxy$,
\item[(iii)] For every $C_0 (X)$-algebra $(A,X,\mu_A)$, the $C_0 ( X \times Y)$-algebra $(A \mint B , X \times Y , \mu_A \otimes \mu_B )$ satisfies property $\fxy$,
\end{enumerate}
If in addition $(B,Y,\mu_B)$ is a continuous $C_0 (Y)$-algebra, then (i) to (iii) are equivalent to:
\begin{enumerate}
\item[(iv)]  for every separable, unital continuous $C(\nhat)$-algebra $(A,\nhat, \mu_A )$, the $C( \nhat \times Y )$-algebra $(A \mint B , \nhat \times Y , \mu_A \otimes \mu_B )$ is continuous,
\item[(v)] for every continuous $C_0 (X)$-algebra $(A,X,\mu_A)$, the $C_0 (X \times Y)$-algebra $(A \mint B , X \times Y , \mu_A \otimes \mu_B )$ is continuous.
\end{enumerate}
\end{corollary}
\begin{proof}
The equivalence of (i),(ii) and (iii) is shown in the proof of~\cite[Proposition 3.1]{blanch_exact}.

To see that (iv) implies (i), we argue by contradiction.  Indeed, suppose that $B$ is inexact, then by Proposition~\ref{p:disc} there is a separable unital $C(\nhat)$-algebra $(A, \nhat , \mu_A )$ with the property that the $C(\nhat )$-algebra $(A \mint B , \nhat , \mu_A \otimes 1)$ is discontinuous.  We will show that the $C(\nhat \times Y)$-algebra $(A \mint B , \nhat \times Y , \mu_A \otimes \mu_B )$ must also be discontinuous.

Let $\phi_A : \Pm (A) \rightarrow \nhat$ and $\phi_B : \Pm (B) \rightarrow Y$ be the base maps corresponding to $\mu_A$ and $\mu_B$ respectively, which are open since  both $A$ and $B$ are continuous.  We will denote by  $\phi_{\alpha} : \Pm (A \mint B ) \rightarrow \nhat \times Y$ and $\overline{\phi_A} : \Pm (A \mint B) \rightarrow \nhat$  the base maps associated with $\mu_A \otimes \mu_B$ and $\mu_A \otimes 1$ respectively.  Note that $\overline{\phi_A}$ is not an open mapping since the $C( \nhat )$-algebra $(A \mint B , \nhat , \mu_A \otimes 1 )$ is not continuous, and that 
\[ \overline{\phi_A} \circ \Phi = p_1 \circ ( \phi_A \times \phi_B ) = p_1 \circ \phi_{\alpha} \circ \Phi \]
on $\Pm (A) \times \Pm (B)$ by Proposition~\ref{p:tensorbundle}(i).

Consider now the  diagram
\[
\xymatrix{
\Pm (A) \times \Pm (B) \ar@{^{(}->}^{\Phi}[r] \ar_{\phi_A \times \phi_B}[d] & \Pm (A \mint B ) \ar_{\phi_{\alpha}}[ld] \ar^{\overline{\phi_A}}[d] \\
\nhat \times Y \ar_{p_1}[r] & \nhat
}
\]  
where $p_1$ is the (open) projection onto the first factor.  To see that the lower triangle of this diagram commutes, note that since $ p_1 \circ \phi_{\alpha} $ agrees with $\overline{\phi_A}$ on $\Phi ( \Pm (A) \times \Pm (B))$, both maps must agree on $\Pm (A \mint B )$ by the uniqueness part of~\cite[Theorem 3.2]{lazar_tensor}.

Then if $\phi_{\alpha}$ were open, this would imply that $\overline{\phi_A} = p_1 \circ \phi_{\alpha}$  were open, which is impossible since $(A \mint B , \nhat , \mu_A \otimes 1 )$ is discontinuous.  In particular, it follows that $(A \mint B, \nhat \times Y , \mu_A \otimes \mu_B )$ is discontinuous.

The fact that (v) implies (iv) is evident.  To see that (i) implies (v), note that if (i) holds then $(A \mint B , X \times Y , \mu_A \otimes \mu_B )$ satisfies property $\fxy$ by the equivalence of (i) and (iii).  Then by Theorem~\ref{c:phi=delta}(iii), $(A \mint B , X \times Y , \mu_A \otimes \mu_B )$ is a continuous $C_0 (X \times Y )$-algebra.

\end{proof}

\section{Quasi-standard C$^{\ast}$-algebras and Hausdorff primitive ideal spaces}
\label{s:exact}
This section is concerned with stability of the class of quasi-standard C$^{\ast}$-algebras under tensor products.  This question was first studied by Archbold in~\cite{arch_cb}, where is was shown that if $A$ and $B$ are quasi-standard and $A \mint B$ satisfies property $\fgl$, then $A \mint B$ is quasi-standard.  We gave a partial converse to this result in~\cite[Theorem 5.2]{m_glimm}.  However, it is clear from Theorem~\ref{t:ineq} that property $\fgl$ is not a necessary condition for $A \mint B$ to be quasi-standard.  

We will show in Theorem~\ref{t:qs} that a quasi-standard C$^{\ast}$-algebra $A$ is exact if and only if $A \mint B$ is quasi-standard for all quasi-standard $B$.  As a related result, we show that a (not necessarily quasi-standard) C$^{\ast}$-algebra $A$ is exact if and only if $A \mint B$ satisfies property $\fgl$ for all C$^{\ast}$-algebras $B$, if and only $A \mint B$ satisfies property $\fmp$ for all C$^{\ast}$-algebras $B$.  The existence of C$^{\ast}$-algebras $A$ and $B$ such that $A \mint B$ does not satisfy properties $\fgl$ and $\fmp$ was previously unknown, thus our result answers a question posed by Archbold~\cite[p. 142]{arch_cb} and Lazar~\cite[p. 250]{lazar_tensor}.

\begin{proposition}
\label{p:phideltaglimm}
Let $A$ and $B$ be C$^{\ast}$-algebras.
\begin{enumerate}
\item[(i)]  If $A \mint B$ satisfies $\fmp$, then $A \mint B$ satisfies $\fgl$,
\item[(ii)] If  $(A , X , \mu_A )$ is a $C_0 (X)$-algebra and $(B , Y , \mu_B)$ a $C_0 (Y)$-algebra, then $A \mint B$ satisfies property $\fgl$ implies that $A \mint B$ satisfies property $\fxy$.
\end{enumerate}

\begin{comment}
Let $(A , X , \phi_A )$ be a $C_0 (X)$-algebra and $(B , Y , \phi_B)$ a $C_0 (Y)$-algebra.  If $\Delta (G_p , G_q ) = \Phi (G_p , G_q )$ for all $(p,q) \in \mathrm{Glimm}(A) \times \mathrm{Glimm}(B)$, then with the notation of Proposition~\ref{p:tensorbundle} we have
\[
\Phi (I_x , J_y ) = \Delta (I_x , J_y ) = K_{x,y}
\]
for all $(x , y ) \in X \times Y$.
\end{comment}
\end{proposition}
\begin{proof}
(i):  We first show that for any C$^{\ast}$-algebra $A$ and $G_1 \in \Gl (A)$,
\[
G_1 = \bigcap \left\{ P \in \MP (A) : P \supseteq G_1 \right\}.
\]
Indeed, since each primitive ideal of $A$ is primal, we necessarily have
\[
G_1 = \bigcap \left\{ P \in \mathrm{Primal} (A) : P \supseteq G_1 \right\} \subseteq \bigcap \left\{ P \in \MP (A) : P \supseteq G_1 \right\}.
\]
Denote by $H$ the ideal on the right, then if the above inclusion were strict, there would be some $Q \in \Pm (A)$ such that $Q \supseteq G_1 $ but $Q \not\supseteq H$.  Let $R$ be a minimal primal ideal of $A$ contained in $Q$, then by~\cite[Lemma 2.2]{arch_som_qs} there is a unique Glimm ideal $G_2$ contained in $R$.  Note $G_1 \neq G_2$ since  otherwise $R$, and hence $Q$, would contain $H$.  This in turn implies that $Q \not\supseteq G_1$, a contradiction.

Since $A \mint B$ satisfies $\fmp$, we have $\Phi = \Delta$ on $\MP (A) \times \MP (B)$, and so~\cite[Theorem 4.1]{arch_cb} shows that $\Delta$ is a homeomorphism of $\MP(A) \times \MP (B)$ onto $\MP (A \mint B )$.  For $(G,H) \in \Gl (A) \times \Gl (B)$ we have $\Delta (G,H) \in \Gl (A \mint B)$~\cite[Theorem 4.8]{m_glimm}, which together with the above remarks gives
\begin{eqnarray*}
\Delta (G,H) & = & \bigcap \left\{ R \in \MP (A \mint B) : R \supseteq \Delta (G,H) \right\} \\
& = & \bigcap \left\{ \Delta (I,J) : (I,J) \in \MP (A) \times \MP (B), \Delta (I,J) \supseteq \Delta (G,H) \right\} \\
\end{eqnarray*}
On the other hand by~\cite[Lemma 2.2]{lazar_tensor} and the first part of the proof, 
\begin{eqnarray*}
\Phi (G,H) & = & \Phi \left( \bigcap \left\{ I \in \MP (A)  : I \supseteq G \right\} , \bigcap \left\{ J \in \MP (B) : J \supseteq H \right\} \right) \\
& = & \bigcap \left\{ \Phi (I,J) : (I,J) \in \MP (A) \times \MP (B), I \supseteq G, J \supseteq H \right\} \\ 
& = & \bigcap \left\{ \Delta (I,J) : (I,J) \in \MP (A) \times \MP (B), I \supseteq G, J \supseteq H \right\}. \\ 
\end{eqnarray*}
Finally, since $\Psi \circ \Delta $ is the identity on $\mathrm{Id}'(A) \times \mathrm{Id}' (B)$ and since $\Psi$ is order preserving, we see that
\[
\Delta (I,J) \supseteq \Delta (G,H) \mbox{ if and only if } I \supseteq G \mbox{ and } J \supseteq H,
\]
from which we conclude that $\Phi (G,H) = \Delta (G,H)$ for all $(G,H) \in \Gl (A) \times \Gl (B)$.  Hence $A \mint B$ satisfies property $\fgl$.

(ii): We will use the notation of Proposition~\ref{p:tensorbundle}.  Note that for all $(x,y) \in X \times Y$, we have $K_{x,y} = \Delta (I_x, J_y)$ by Proposition~\ref{p:tensorbundle}(iv), so we will show that $K_{x,y} = \Phi (I_x,J_y)$. 
Let $\psi_A , \psi_B$ and $\psi_{\alpha}$ denote the continuous maps on the Glimm spaces of $A,B$ and $A \otimes_{\alpha} B$ induced by the base maps $\phi_A,\phi_B$ and $\phi_{\alpha}$ respectively.  
We first show that $\psi_{\alpha} \circ \Delta = \psi_A \times \psi_B$.  Indeed, for all $(P,Q) \in \Pm (A) \times \Pm (B)$ we have
\[
(\phi_A \times \phi_B )(P,Q) = (\phi_{\alpha} \circ \Phi )(P,Q)
\]
by Proposition~\ref{p:tensorbundle}(i).  Hence by the definitions of $\psi_A,\psi_B$ and $\psi_{\alpha}$,
\[
(\psi_A \times \psi_B ) \circ ( \rho_A \times \rho_B ) (P,Q) = (\psi_{\alpha} \circ \rho_{\alpha} \circ \Phi) (P,Q).
\]
Following the first paragraph of the proof of~\cite[Theorem 4.8]{m_glimm}, we see that $\rho_{\alpha} \circ \Phi = \Delta \circ (\rho_A \times \rho_B)$, which shows that
\[
(\psi_A \times \psi_B ) ( \rho_A (P) , \rho_B (Q) ) = (\psi_{\alpha} \circ \Delta ) ( \rho_A (P) , \rho_B (Q) )
\]
for all $(P,Q) \in \Pm (A) \times \Pm (B)$.

By Lemma~\ref{l:capglimm},
\[
K_{x,y} = \bigcap \{ G \in \mathrm{Glimm}(A \otimes_{\alpha} B ) : \psi_{\alpha} (G) = (x,y) \}.
\]
Any such $G \in \mathrm{Glimm}(A \otimes_{\alpha} B )$ is the image $\Delta (G_p , G_q)$ of a pair of Glimm ideals of $A $ and $B$ by~\cite[Theorem 4.8]{m_glimm}.  Together with the fact that $\psi_{\alpha} \circ \Delta = \psi_A \times \psi_B$, this gives
\[
K_{x,y} = \bigcap \{ \Delta (G_p , G_q ) : (p,q) \in \mathrm{Glimm}(A) \times \mathrm{Glimm}(B), \psi_A (p) = x \mbox{ and } \psi_B (q) = y \}.
\]
Using $\Phi = \Delta$ on $\mathrm{Glimm}(A) \times \mathrm{Glimm}(B)$,~\cite[Lemma 2.2]{lazar_tensor} shows that
\begin{eqnarray*}
K_{x,y} & = & \bigcap \{ \Phi (G_p , G_q ) : \psi_A (p) = x , \psi_B (q) = y \} \\
& = & \Phi \left( \cap \{ G_p : \psi_A (p)=x \}  , \cap \{ G_q : \psi_B (q) = y \}  \right) \\
& = & \Phi (I_x, J_y),
\end{eqnarray*}
where the final equality follows from Lemma~\ref{l:capglimm}.  Hence $A \mint B$ satisfies property $\fxy$.
\end{proof}

\begin{proposition}
\label{p:counter}
Let $B$ be an inexact C$^{\ast}$-algebra.  Then there is a separable C$^{\ast}$-algebra $A$ with $\Pm (A)$ Hausdorff such that
\begin{enumerate}
\item[(i)] there is a pair $(G,H) \in \Gl (A) \times \Gl (B)$ with $\Delta (G,H) \subsetneq \Phi (G,H)$,
\item[(ii)] there is a pair $(I,J) \in \MP (A) \times \MP (B)$ with $\Delta (I,J) \subsetneq \Phi (I,J)$,
\item[(iii)] $\Pm (A \mint B )$ is non-Hausdorff,
\item[(iv)] The complete regularisation map $\rho_{\alpha} : \Pm (A \mint B ) \rightarrow \Gl (A \mint B)$ is not open.
\end{enumerate}
\end{proposition}
\begin{proof}
To prove (i), let $(A , \nhat , \mu_A )$ be the continuous $C( \nhat )$-algebra  constructed in Corollary~\ref{c:disc}, so that the $C( \nhat )$-algebra $(A \mint B , \nhat , \mu_A \otimes 1 )$ is discontinuous. We regard $B$ as a (continuous) $C_0 (Y)$-algebra over a one-point space $Y = \{ y \}$ as in Remark~\ref{r:tensorbundle}(ii).   Then since $(A \mint B , \nhat , \mu_A \otimes 1 )$ is discontinuous, it follows from Theorem~\ref{c:phi=delta}(iii) that $A \mint B$ does not satisfy property $\fxy$.  Hence by Proposition~\ref{p:phideltaglimm}(ii), it must follow that $\Phi \neq \Delta $ on $\Gl (A) \times \Gl (B)$. (ii) is immediate from (i) and Proposition~\ref{p:phideltaglimm}(i).

(iii):  Note that if $\Pm (B)$ is non-Hausdorff, then the same is true of $\Pm (A) \times \Pm (B)$. Since $\Phi$ maps $\Pm (A) \times \Pm (B)$ homeomorphically onto its  image in $\Pm (A \mint B)$~\cite[lemme 16]{wulfsohn}, it then follows that the latter must also be non-Hausdorff.

Suppose now that $\Pm (B)$ is Hausdorff, then $\Pm (B) = \Gl (B)$.  Since $\Pm (A) = \Gl (A)$ also, (i) implies that there are $(P,Q) \in \Pm (A) \times \Pm (B)$ such that $\Delta (P,Q) \subsetneq \Phi (P,Q)$. It follows that there is $R \in \Pm (A \mint B )$ such that $R \supseteq \Delta (P,Q)$ but $R \neq \Phi (P,Q)$.  Since by~\cite[Theorem 4.8]{m_glimm}, $\Delta (P,Q)$ is a Glimm ideal of $A \mint B$, we have $R \approx \Phi (P,Q)$, so that $\Pm (A \mint B )$ is non-Hausdorff.

(iv):  If $\rho_B : \Pm (B) \rightarrow \Gl (B)$ is not an open map, then since $A$ is unital, $\rho_{\alpha}$ is not open by~\cite[Corollary 5.6]{m_glimm}.  Thus we will assume that $\rho_B$ is open.

Since $A \mint B$ is a discontinuous $C(\hat{\mathbb{N}} )$-algebra, the continuous mapping $\phi_{\alpha} : \Pm (A \mint B) \rightarrow \hat{\mathbb{N}}$ is not open.  Moreover, $\phi_{\alpha}$ is the unique extension of $\phi_A \circ p_1 : \Pm (A) \times \Pm (B) \rightarrow \hat{\mathbb{N}}$ to $\Pm (A \mint B)$, where $p_1$ is the projection onto the first factor.  We will denote by $\psi_A : \Gl (A) \rightarrow \nhat$ the canonical homeomorphism, and by $\psi_{\alpha} : \Gl (A \mint B ) \rightarrow \nhat $ the map induced by $\phi_{\alpha}$ given by the universal property of the complete regularisation, so that $\phi_{\alpha} = \psi_{\alpha} \circ \rho_{\alpha}$.

 Since $\Pm(A)$ is compact,  the complete regularisation of $\Pm (A) \times \Pm (B)$ is canonically identified with $\Gl (A) \times \Gl (B)$ with the product topology~\cite[Proposition 1.9]{m_glimm}. In particular, this implies that $\Psi : \Gl (A \mint B ) \rightarrow \Gl (A) \times \Gl (B)$ is a homeomorphism by~\cite[Theorem 4.8]{m_glimm}. Let $\tilde{p_1} : \Gl (A) \times \Gl (B) \rightarrow \Gl (A)$ be the projection onto the first factor and consider now the diagram
\[
\xymatrix{
\Pm (A \mint B ) \ar^{\rho_{\alpha}}[r] \ar^{\phi_{\alpha}}[d] & \Gl (A \mint B) \ar^{\Psi}[r] & \Gl(A) \times \Gl (B) \ar^{\tilde{p_1}}[d] \\
\hat{\mathbb{N}} & & \Gl(A). \ar^{\psi_A}[ll] \\
}
\]
Then by~\cite[Theorem 2.2 and Theorem 4.8(ii)]{m_glimm} we have $\Psi \circ \rho_{\alpha} \circ \Phi = \rho_A \times \rho_B$ on $\Pm (A) \times \Pm (B)$, so that for all $(P,Q) \in \Pm (A) \times \Pm (B)$, 
\begin{eqnarray*}
(\psi_A \circ \tilde{p}_1 \circ \Psi \circ \rho_{\alpha} ) ( \Phi (P,Q) ) & = & (\psi_A \circ \tilde{p}_1 ) ( \rho_A (P), \rho_B (Q) ) \\
& = & (\psi_A \circ \rho_A ) (P) \\
& = & \phi_A (P) \\
& = & ( \phi_{\alpha} ) ( \Phi (P,Q) ),
\end{eqnarray*}
the final equality holding by Proposition~\ref{p:tensorbundle}(i).  Since $\Phi ( \Pm (A) \times \Pm (B))$ is dense in $\Pm (A \mint B )$, it follows by continuity that the above diagram commutes.

Note that $\Psi, \tilde{p}_1$ and $\psi_A$ are all open mappings. Thus if $\rho_{\alpha} $ were open, it would follow that $\phi_{\alpha}$ were, which would imply that $(A \mint B, \nhat , \mu_A \otimes 1)$ were a continuous $C(\hat{\mathbb{N}})$-algebra, a contradiction.  
\end{proof}

\begin{theorem}
\label{t:exact}
The following conditions on a C$^{\ast}$-algebra $B$ are equivalent:
\begin{enumerate}
\item[(i)] $B$ is exact,
\item[(ii)] $A \mint B$ satisfies property $\fgl$ for all C$^{\ast}$-algebras $A$,
\item[(iii)] $A \mint B$ satisfies property $\fmp$ for all C$^{\ast}$-algebras $A$.
\end{enumerate}
\end{theorem}
\begin{proof}
(i)$\Rightarrow$(ii) and (iii): if $B$ is exact then $A \mint B$ satisfies property (F) for all $A$, hence $A \mint B$ satisfies properties $\fgl$ and $\fmp$ for all C$^{\ast}$-algebras $A$.

To see that (ii) (resp. (iii)) implies (i), note that if $B$ is inexact then there is by Proposition~\ref{p:counter}(i) (resp. (ii)) a C$^{\ast}$-algebra $A$ for which $A \mint B$ does not satisfy property $\fgl$ (resp. $\fmp$).
\end{proof}

It was shown in~\cite{arch_cb} that if $A \mint B$ satisfies property $\fmp$, then $\Delta$ (equivalently, $\Phi$) maps $\MP (A) \times \MP (B)$ homeomorphically onto $\MP (A \mint B )$.  The following corollary shows that if $B$ is inexact and quasi-standard then this map may fail to be a homeomorphism.

\begin{corollary}
\label{c:mp}
Let $B$ be an inexact, quasi-standard C$^{\ast}$-algebra.  Then there is a quasi-standard C$^{\ast}$-algebra $A$ for which the restriction of $\Delta$ to $\MP (A) \times \MP (B)$ is not a homeomorphism of this space onto $\MP (A \mint B )$.
\end{corollary}
\begin{proof}
Again, let $A$ be the C$^{\ast}$-algebra constructed in Corollary~\ref{c:disc}, so that $A \mint B$ is not quasi-standard by Proposition~\ref{p:counter}(iv).  Since both $A$ and $B$ are quasi-standard, we have $\Gl (A) = \MP (A)$ and $\Gl (B) = \MP (B)$, both as sets and topologically~\cite[Theorem 3.3 (iii)]{arch_som_qs}. By~\cite[Corollary 2.3 and Theorem 4.8]{m_glimm}, $\Delta$ is a homeomorphism of $\Gl(A) \times \Gl (B)$ onto $\Gl (A \mint B )$.  Thus if $\Delta$  were also a homeomorphism of $\MP(A) \times \MP (B)$ onto $\MP (A \mint B )$, it would follow that $\Gl (A \mint B ) = \MP (A \mint B )$, both as sets and topologically. This would imply that $A \mint B$ is quasi-standard by~\cite[Theorem 3.3 (iii)$\Rightarrow$(i)]{arch_som_qs}, which is a contradiction.

\end{proof}

\begin{example}
Let $B$ be a primitive, inexact C$^{\ast}$-algebra, for example $B = B(H)$~\cite{wassermann} or $B = C^{\ast} ( \mathbf{F}_2 )$~\cite{wassermann_groups} (the full group C$^{\ast}$-algebra of the free group on two generators), so that $\MP (B) = \{ 0 \}$.  Then Corollary~\ref{c:mp}  gives a C$^{\ast}$-algebra $A$ with $\Pm (A) = \MP (A) = \{ I_n : n \in \nhat \}$ for which $\Delta$ is not a homeomorphism of $\MP (A) \times \MP (B)$ onto $\MP (A \mint B )$.

We remark that by Theorem~\ref{t:kirch_wass1}(iii), $\Delta (I_n , \{ 0 \} ) = \Phi (I_n , \{ 0 \} )$ for all $n \in \mathbb{N}$, so that $\Delta (I_n , \{ 0 \} ) \in \Pm (A \mint B)$ for all such $n$.  By~\cite[Proposition 4.5]{arch_primal}, $\Delta (I_n , \{ 0 \} ) \in \MP (A \mint B)$ for all $n \in \mathbb{N}$.  On the other hand, it is not clear whether or not the Glimm ideal $\Delta (I_{\infty} , \{ 0 \} )$ is a primal ideal of $A \mint B$.
\end{example}

\begin{theorem}
\label{t:qs}
Let $B$ be a C$^{\ast}$-algebra.  Then 
\begin{enumerate}
\item[(i)] If $\rho_B: \Pm (B) \rightarrow \Gl (B)$ is an open map, then $B$ is exact if and only if $\rho_{\alpha} : \Pm (A \mint B ) \rightarrow \Gl (A \mint B )$ is open for every C$^{\ast}$-algebra $A$ with $\rho_A : \Pm (A) \rightarrow \Gl (A)$ open.
\item[(ii)]  If $B$ is quasi-standard, then $B$ is exact if and only if $A \mint B$ is quasi-standard for every quasi-standard C$^{\ast}$-algebra $A$.
\item[(iii)]  If $\Pm (B)$ is Hausdorff, then $B$ is exact if and only if $\Pm (A \mint B )$ is Hausdorff for every  C$^{\ast}$-algebra $A$ with $\Pm (A)$ Hausdorff.
\end{enumerate}
\end{theorem}
\begin{proof}
(i): If $B$ is exact then $A \mint B$ satisfies property (F), hence property $\fgl$, for all C$^{\ast}$-algebras $A$.  Hence $\rho_{\alpha}$ is open for all C$^{\ast}$-algebras $A$ with $\rho_A$ open by~\cite[Theorem 5.2]{m_glimm}.

Conversely if $B$ is inexact, then by Proposition~\ref{p:counter}(iv) there is a C$^{\ast}$-algebra $A$ with $\Pm(A)$ Hausdorff, hence $\rho_A = \mathrm{id}$ open, such that $\rho_{\alpha}$ is not open.

(ii): If $B$ is exact, then it follows from~\cite[Corollary 2.5]{kaniuth} that $A \mint B$ is quasi-standard for every quasi-standard $A$.

Conversely, if $B$ is inexact then by Proposition~\ref{p:counter}(iv) there is a quasi-standard C$^{\ast}$-algebra $A$ for which $A \mint B$ is not quasi-standard.

(iii): We will make use of the fact that if $A$ is a C$^{\ast}$-algebra such that either $\Pm (A)$ or $\Fac (A)$ is Hausdorff, then $\Pm (A) = \Fac (A) = \Pme (A)$, see e.g.~\cite[p. 474]{blanch_kirch}.  

Suppose that $B$ is exact and that $A$ is a C$^{\ast}$-algebra with $\Pm (A)$ Hausdorff. Then since $A \mint B$ has property (F),~\cite[Proposition 5.1]{lazar_tensor} shows that $\Delta$ is a homeomorphism of $\Fac (A) \times \Fac (B) = \Pm (A) \times \Pm (B)$ onto $\Fac (A \mint B )$.   Hence $\Fac (A \mint B )$ is Hausdorff, and thus the same is true of $\Pm (A \mint B)$.

Conversely, if $B$ is inexact then by Proposition~\ref{p:counter}(iii) there is a separable C$^{\ast}$-algebra $A$ with $\Pm(A)$ Hausdorff for which $\Pm (A \mint B)$ is non-Hausdorff.

\end{proof}
\begin{example}
Let $M = \prod_{n \geq 1} M_n ( \mathbb{C} )$, so that $M$ is quasi-standard as in Theorem~\ref{t:ineq}. Moreover, $M$ is inexact by~\cite[Theorem 1.1]{kirch_fubini}. Then by Proposition~\ref{p:counter}(iv), there is a separable unital C$^{\ast}$-algebra $A$ with Hausdorff primitive ideal space $\Pm (A)$  homeomorphic to $\nhat$ such that $A \mint M$ is not quasi-standard.  In particular, the assumption that the zero ideal of the C$^{\ast}$-algebra $B$ of Theorem~\ref{t:ineq} is prime cannot be dropped.
\end{example}

We will give an analogous result for maximal tensor products of (unital) C$^{\ast}$-algebras in Theorem~\ref{t:maxt}.  The following proposition gives some (known) properties of the Dauns-Hofmann representation of the maximal tensor product of two unital C$^{\ast}$-algebras.
\begin{proposition}
\label{p:maxt}
Let $A$ and $B$ be unital C$^{\ast}$-algebras, then
\begin{enumerate}
\item[(i)] the map $(G,H) \mapsto G \maxt B + A \maxt H$ is a homeomorphism of $\Gl (A) \times \Gl (B)$ onto $\Gl (A \maxt B )$,
\item[(ii)] the Glimm quotients of $A \maxt B$ are canonically $\ast$-isomorphic to $(A/G) \maxt (B/H)$ for $(G,H) \in \Gl (A) \times \Gl (B)$,
\item[(iii)] The Dauns-Hofmann representation of $A \maxt B$ defines canonically an upper-semicontinuous C$^{\ast}$-bundle
\[
\left( \Gl (A) \times \Gl (B) , A \maxt B , \pi_G \maxt \sigma_H : A \maxt B \rightarrow (A/G) \maxt (B / H ) \right),
\]
where for $(G,H) \in \Gl(A) \times \Gl (B)$ , $\pi_G : A \rightarrow A/G$ and $\sigma_H : B \rightarrow B/ H$ are the quotient maps.
\end{enumerate}
\end{proposition}
\begin{proof}
(i) is shown in~\cite[p. 304]{kaniuth}, and (ii) follows from (i) and~\cite[Proposition 3.15]{blanch_def}. Since $A$ and $B$ are unital, $\Gl (A) \times \Gl (B)$ is compact, and so (iii) is then immediate from the Dauns-Hofmann Theorem (Theorem~\ref{t:dh}).
\end{proof}

\begin{theorem}
\label{t:maxt}
Let $B$ be a unital quasi-standard C$^{\ast}$-algebra.  Then $B$ is nuclear if and only if $A \maxt B$ is quasi-standard for every quasi-standard C$^{\ast}$-algebra $A$.
\end{theorem}
\begin{proof}
If $B$ is nuclear, then $B$ is exact, so that for any C$^{\ast}$-algebra $A$ we have that $A \maxt B = A \mint B$ has property (F).  Thus for all quasi-standard $A$, $A \maxt B$ is quasi-standard by~\cite[Corollary 2.5]{kaniuth}.

Conversely, suppose that $B$ is non-nuclear.  For each $q \in \Gl (B)$ let $G_q$ be the corresponding Glimm ideal of $B$ and denote by $ ( \Gl (B) , B , \sigma_q : B \rightarrow B_q )$ the corresponding continuous C$^{\ast}$-bundle over $\Gl (B)$, where $B_q = B / G_q$ for all $ q \in \Gl (B)$. By~\cite[Proposition 3.23]{blanch_def}, there is $p \in \Gl (A)$ for which $B_p$ is non-nuclear.  As in the proof of~\cite[Theorem 3.2]{kirch_wass}, one can construct a Hilbert space $H$, a unital C$^{\ast}$-subalgebra $C \subseteq B(H)$ and $t = \sum_{i=1}^{\ell} r_i \otimes s_i \in C \odot B_p$ such that $\norm{t}_{C \maxt B_p} > \norm{t}_{B(H) \maxt B_p}$.

Denote by $A$ the C$^{\ast}$-algebra of sequences $(T_n) \subset B(H)$ such that $T_n$ converges in norm to some element $T \in C$.   Then, as in the proof of~\cite[Proposition 3.6]{arch_som_qs}, $A$ is quasi-standard, $\Gl (A)$ is homeomorphic to $\hat{\mathbb{N}}$, and  the Glimm quotients $A_n$ of $A$ are given by
\[
A_n = \left\{ \begin{array}{c c l}
B(H) & \mbox{ if } & n \in \mathbb{N} \\
C & \mbox{ if } & n = \infty 
\end{array}
\right.
\]
 We denote by $(\hat{\mathbb{N}} , A , \pi_n : A \rightarrow A_n )$ the corresponding continuous C$^{\ast}$-bundle over $\hat{\mathbb{N}}$.   By Proposition~\ref{p:maxt} we may identify $\Gl (A \maxt B ) = \hat{\mathbb{N}} \times \Gl (B)$, and the Glimm quotients of $A \maxt B$ are isomorphic to $A_n \maxt B_p$ for $(n,p) \in \hat{\mathbb{N}} \times \Gl (B)$.

For $1 \leq i \leq \ell$ choose $\overline{r_i} \in A$ and $\overline{s_i} \in B$ such that $\pi_n (\overline{r_i}) = r_i$ for all $n \in \nhat$, and $\sigma_{p} (\overline{s_i} ) = s_i$.  Then setting $\overline{t} = \sum_{i=1}^{\ell} \overline{r_i} \otimes \overline{s_i} \in A \odot B$, we have $\pi_{\infty} \otimes \sigma_G (\overline{t})=t$.  Then we have $\norm{ (\pi_n \maxt \sigma_p) ( \overline{t})} = \norm{t}_{B(H) \maxt B}$ for $n \in \mathbb{N}$, while $\norm{ (\pi_{\infty} \maxt \sigma_p )( \overline{t} )} = \norm{t}_{C \maxt B}$.   Since $(n,p) \rightarrow (\infty , p )$ in $\hat{\mathbb{N}} \times \Gl (B)$, it follows that $(n,q) \mapsto \norm{(\pi_n \maxt \sigma_q)(t)}$ is discontinuous at $(\infty , p )$.  In particular, $A \maxt B$ is not quasi-standard.
\end{proof}

\textbf{Acknowledgement:}  I am grateful to R.J. Archbold for detailed comments on an earlier draft of this work, and also my supervisor R.M. Timoney for his helpful suggestions.

\bibliography{refs2}
\bibliographystyle{rt}
\end{document}